\documentclass[final,leqno,onefignum,onetabnum]{article}

\usepackage{amsmath}
\usepackage{amsthm}
\usepackage{latexsym}
\usepackage{amssymb}	
\usepackage{lipsum}
\usepackage{amsfonts}
\usepackage{graphicx}
\usepackage{epstopdf}
\usepackage{hyperref}
\usepackage{geometry}
\usepackage{algorithm,algorithmic}

\newcommand{\esp}{\mathbb{E}}

\newcommand{\probn}{\mathbf{P}}

\newcommand{\alg}{\mathcal{F}}

\DeclareMathOperator*{\solset}{S}

\newcommand{\re}{\mathbb{R}}

\DeclareMathOperator*{\dist}{d}
\DeclareMathOperator*{\diam}{\mathcal{D}}
\DeclareMathOperator*{\argmin}{argmin}

\DeclareMathOperator*{\cov}{\mbox{cov}}

\newtheorem{assump}{Assumption}

\newtheorem{lemma}{Lemma}
\newtheorem{theorem}{Theorem}
\newtheorem{corollary}{Corollary}
\newtheorem{proposition}{Proposition}
\theoremstyle{definition}
\newtheorem{example}{Example}
\theoremstyle{definition}
\newtheorem{remark}{Remark}
\newcommand{\vertiii}[1]{{\left\vert\kern-0.4ex #1 
\kern-0.4ex\right\vert}}
\newcommand{\Lpnorm}[1]{\vertiii{\,#1\,}_{p}}

\newcommand{\Lnorm}[1]{\vertiii{\,#1\,}_{2}}

\begin{document}

\title{On variance reduction for stochastic smooth convex optimization with multiplicative noise} 
		
\author{Alejandro Jofr\'e, Center for Mathematical Modeling (CMM) \& DIM, \href{}{ajofre@dim.uchile.cl},\and
Philip Thompson, Center for Mathematical Modeling (CMM), \href{}{pthompson@uchile.dim.cl}
}
\maketitle

\begin{abstract}
We propose dynamic sampled stochastic approximation (SA) methods for stochastic optimization with a heavy-tailed distribution (with finite 2nd moment). The objective is the sum of a smooth convex function with a convex regularizer. Typically, it is assumed an oracle with an upper bound $\sigma^2$ on its variance (OUBV). Differently, we assume an oracle with \emph{multiplicative noise}. This rarely addressed setup is more aggressive but realistic, where the variance may not be bounded.  Our methods achieve optimal iteration complexity and (near) optimal oracle complexity. For the smooth convex class, we use an accelerated SA method a la FISTA which achieves, given tolerance $\epsilon>0$, the optimal iteration complexity of $\mathcal{O}(\epsilon^{-\frac{1}{2}})$ with a near-optimal oracle complexity of $\mathcal{O}(\epsilon^{-2})[\ln(\epsilon^{-\frac{1}{2}})]^2$. This improves upon Ghadimi and Lan [\emph{Math. Program.}, 156:59-99, 2016] where it is assumed an OUBV. For the strongly convex class, our method achieves optimal iteration complexity of $\mathcal{O}(\ln(\epsilon^{-1}))$ and optimal oracle complexity of $\mathcal{O}(\epsilon^{-1})$. This improves upon Byrd et al. [\emph{Math. Program.}, 134:127-155, 2012] where it is assumed an OUBV. In terms of variance, our bounds are local: they depend on variances $\sigma(x^*)^2$ at solutions $x^*$ and the per unit distance multiplicative variance $\sigma^2_L$. For the smooth convex class, there exist policies such that our bounds resemble those obtained if it was assumed an OUBV with $\sigma^2:=\sigma(x^*)^2$. For the strongly convex class such property is obtained exactly if the condition number is estimated or in the limit for better conditioned problems or for larger initial batch sizes. In any case, if it is assumed an OUBV, our bounds are thus much sharper since typically $\max\{\sigma(x^*)^2,\sigma_L^2\}\ll\sigma^2$.
\end{abstract}

\section{Introduction}

We consider methods for convex optimization problems where only noisy first-order information is assumed. This setting includes problems in signal processing and empirical risk minimization for machine learning  \cite{bach:moulines2011,sra:nowozin:wright2012,bottou:curtis:nocedal2016},
stochastic optimization and finance \cite{shapiro:dent:rus2009,nem:jud:lan:shapiro2009} and simulation optimization \cite{fu2014}. In such problems, we have a closed convex set $X\subset\re^d$, a distribution $\probn$ over a sample space $\Xi$ and a measurable function $F:X\times\Xi\rightarrow\re$ satisfying
\begin{equation}\label{equation:expected:valued:objective}
f(x):=\esp F(x,\xi)=\int_\Xi F(x,\xi)\dist\probn(\xi),\quad (x\in X),
\end{equation}
where for almost every (a.e.) $\xi\in\Xi$, $F(\cdot,\xi)$ is a continuously differentiable convex function for which $F(x,\cdot)$ and $\nabla F(x,\cdot)$ are integrable. The stochastic optimization problem is then to solve 
\begin{eqnarray}\label{problem:min:intro}
\min_{x\in X}f(x).
\end{eqnarray}
The challenge aspect of stochastic optimization, when compared to deterministic optimization, is that the expectation \eqref{equation:expected:valued:objective} cannot be evaluated.\footnote{Typical reasons are: a sample space with high dimension requiring Monte Carlo evaluation, no knowledge of the distribution $\probn$ or, even worse, no knowledge of a close form for $F$.} However, a practical assumption is that the decision maker have access to samples drawn from the distribution $\probn$. 

Two different methodologies exist for solving \eqref{equation:expected:valued:objective}-\eqref{problem:min:intro} when samples $\{\xi_j\}_{j=1}^N$ of the distribution $\probn$ is available. The \emph{Sample Average Approximation} (SAA) methodology is to solve the problem
\begin{eqnarray}\label{problem:SAA}
\min_{x\in X}\left\{\widehat F_N(x):=\frac{1}{N}\sum_{j=1}^NF(\xi_j, x)\right\},
\end{eqnarray}
by resorting to a chosen algorithm. See for instance \cite{shapiro:dent:rus2009,nem:jud:lan:shapiro2009}
for such kind of approach in stochastic optimization based on Monte Carlo simulation. Such methodology is also the case of Empirical Risk Minimization (ERM) in statistical machine learning where $\probn$ is unknown and a limited number of samples is acquired by measurements. Note that \eqref{problem:SAA} itself is of the form \eqref{equation:expected:valued:objective}-\eqref{problem:min:intro} with the empirical distribution $\widehat\probn_N:=\frac{1}{N}\sum_{j=1}^N\delta_{\xi_j}$, where $\delta_\xi$ denotes the Dirac measure at the point $\xi\in\Xi$.

A different methodology is the \emph{Stochastic Approximation} (SA) approach where the samples are accessed in an interior and online fashion: a deterministic version of an algorithm is chosen and samples are used whenever the algorithm requires gradients at the current or previous iterates \cite{nem:jud:lan:shapiro2009,lan2012}. In this setting the mechanism to access $F$ via samples of $\probn$ is usually named a \emph{stochastic oracle} (SO). Precisely, given an input $x\in X$ and an independent identically distributed (i.i.d.) sample $\{\xi_j\}$ drawn from $\probn$ (also independent of the input $x$), the SO outputs unbiased gradients $\{\nabla F(x,\xi_j)\}$, that is, satisfying $\esp[\nabla F(x,\xi_j)]=\nabla f(x)$ for all $j$.

The SA methodology was first proposed by Robbins and Monro in \cite{robbins:monro1951} for problem \eqref{equation:expected:valued:objective}-\eqref{problem:min:intro}
when $f$ is a smooth strongly convex function under specific
conditions. In the unconstrained case, the method takes the form
\begin{equation}\label{equation:SGD}
x^{t+1}:=x^t-\alpha_t\nabla F(x^t,\xi^t),
\end{equation}
where $\alpha_t$ is a positive stepsize and $\xi^t$ is a sample from $\probn$. The above method is one SA version of the gradient descent method, known as Stochastic Gradient method (SG). This methodology was then extensively explored in numerous works spanning the communities of statistics and stochastic approximation,  stochastic optimization and machine learning. We refer to \cite{chung1954,dvoretzky1956,nem:jud:lan:shapiro2009,bach:moulines2011,sra:nowozin:wright2012,%
ghadimi:lan2012,ghadimi:lan2013}
for further references. More recently, the SA methodology was also analyzed for stochastic
variational inequalities (VI). See e.g. \cite{jiang:xu2008,juditsky:nemirovski:tauvel2011,%
iusem:jofre:thompson2015,iusem:jofre:oliveira:thompson2017,iusem:jofre:oliveira:thompson2016,%
balamurugan:bach2016} and references therein. 
VI is a framework which generalizes unconstrained system of equations and the first order necessary condition of constrained minimization (including a broader class of problems such as equilibrium, saddle-point and complementarity problems). 

In this work we consider SA methods for solving constrained and regularized stochastic \emph{smooth convex} optimization problems (CO) of the form
\begin{eqnarray}\label{problem:min:composite}
g^*:=\min_{x\in X}\left\{g(x):=f(x)+\varphi(x)\right\},
\end{eqnarray}
where $X\subset\re^d$ is closed and convex, $f:X\rightarrow\re$ is a smooth convex function satisfying \eqref{equation:expected:valued:objective} and $\varphi:X\rightarrow\re^d$ is a (possibly nonsmooth) convex function. By smoothness we mean $f$ is a differentiable function satisfying, for some $L>0$,
\begin{eqnarray}\label{equation:smoothness:intro}
\Vert\nabla f(x)-\nabla f(y)\Vert\le L\Vert x-y\Vert, \quad\forall x,y\in X,
\end{eqnarray}
where $\Vert\cdot\Vert$ denotes the Euclidean norm. The function $\varphi$ is suppose to be a known simple function so that proximal evaluations are easy to compute (See Section \ref{section:preliminaries:notation}). The above set-up includes many problems in stochastic optimization, simulation optimization and, in particular, inference problems in machine learning and signal processing \cite{nem:jud:lan:shapiro2009,bottou:curtis:nocedal2016,beck:teboulle2009}. The function $\varphi$ is often used in applications to induce regularization and parsimony (such as sparsity). Examples include $\varphi:=\lambda\Vert\cdot\Vert^2$ (as in \emph{ridge regression}), $\varphi:=\lambda\Vert\cdot\Vert_1$ (as in the Lasso estimator) or $\varphi:=\lambda\Vert\cdot\Vert^2+\gamma\Vert\cdot\Vert_1$ (as in \emph{elastic net}) for some $\lambda,\gamma>0$ where $\Vert\cdot\Vert_1$ denotes the the $\ell_1$-norm (see, e.g. \cite{xiao:zhang2014}). We will denote the solution set of \eqref{problem:min:composite} by S$(f,\varphi)$.

We will also consider the special subclass where $f$ is \emph{smooth $c$-strongly convex}, i.e., satisfying \eqref{equation:smoothness:intro} and for some $c>0$,
\begin{eqnarray}\label{equation:strong:convexity}
f(x)+\langle\nabla f(x),y-x\rangle+\frac{c}{2}\Vert y-x\Vert^2\le f(y),\quad\forall x,y\in X.
\end{eqnarray}
In this case, problem \eqref{problem:min:composite} has an unique solution. 
In order to present the main ideas with respect to variance reduction, we refrain from considering non-Euclidean geometries which will be treated in future work. As usual, we make the following assumption on the sampling resource.
\begin{assump}[i.i.d. sampling]\label{assump:iid:sampling}
In all our methods, the samples drawn from the distribution $\probn$ and used along the chosen algorithm are i.i.d..
\end{assump}

\subsection{Oracle assumptions}\label{subsection:oracle:assumptions}

In stochastic optimization, the assumption on the \emph{stochastic oracle's variance} is as important as the class of smoothness of the objective since both have consequences in obtaining surrogate models and on condition numbers of the problem. Define, for every $x\in X$, the pointwise oracle's standard deviation by
$
\sigma(x):=\Lnorm{\Vert\nabla F(x,\xi)-\nabla f(x)\Vert},
$
where $\Lnorm{q(\xi)}:=\sqrt{\esp[q(\xi)^2]}$ denotes the $\mathsf{L}^2$-norm of the random variable $q:\Xi\rightarrow\re$.

A reasonable hypothesis, used since the seminal work of Robbins and Monro \cite{robbins:monro1951}, is to assume a SO with an \emph{uniformly bounded variance} over the feasible set $X$, i.e., that there exists $\sigma>0$, such that
\begin{eqnarray}\label{equation:uniform:variance:intro}	
\sup_{x\in X}\esp\left[\Vert\nabla F(x,\xi)-\nabla f(x)\Vert^2\right]\le\sigma^2.
\end{eqnarray}
Condition \eqref{equation:uniform:variance:intro} is a \emph{global variance} assumption on the noise when using stochastic approximation. Assumption \eqref{equation:uniform:variance:intro} is valid in important instances, e.g., when the feasible set $X$ is compact or when an \emph{uniform additive noise} is assumed, that is, for some centered random variable $\epsilon\in\re^d$ with $\esp[\Vert\epsilon\Vert^2]\le\sigma^2$, for a.e. $\xi\in\Xi$,
\begin{equation}\label{equation:additive:noise}
\nabla F(\xi,x)=\nabla f(x)+\epsilon(\xi),\quad\forall x\in X.
\end{equation}
Property \eqref{equation:additive:noise} is a reasonable assumption in many important ERM problems, such as instances of the least squares regression problem (LSR). Note that property \eqref{equation:additive:noise}, although more structured, allows unconstrained optimization ($X=\re^d$). 

Property \eqref{equation:additive:noise} is not a reasonable assumption in problems where the noise is dependent on the point of the feasible set. In  that case, property \eqref{equation:uniform:variance:intro} is an important generalization of \eqref{equation:additive:noise} assumed in most stochastic optimization methods. However, it implicitly assumes compactness of $X$. This has two drawbacks. The first is that it rules out unconstrained minimization problems where  \eqref{equation:additive:noise} is not satisfied. This includes many stochastic and simulation optimization problems as well as LSR without additive noise, a more aggressive but relevant condition in ERM. The second is that, even if  \eqref{equation:uniform:variance:intro} does hold, $\sigma^2$ may be very large. The reason is that, in case of multiplicative noise, \emph{$\sigma(\cdot)^2$ is typically coercive over} $X$ and, hence, $\sigma^2$ grows with the diameter of $X$ (see Example \ref{example:1} in the following).

In this work we will consider the following assumption.
\begin{assump}[Oracle with multiplicative noise]\label{assump:oracle:multiplicative:noise}
There exist $x^0\in X$ such that $\sigma(x^0)<\infty$ and $\sigma_L>0$ such that\footnote{The convergence theory we present would work if \eqref{equation:nonuniform:variance:intro} is satisfied for just one $x^*\in X$. However, more uniform bounds are obtained under Assumption \ref{assump:oracle:multiplicative:noise} which is satisfied when \eqref{equation:lipschitz:random:intro} holds.}
\begin{equation}\label{equation:nonuniform:variance:intro}
\sigma(x)\le\sigma(x^*)+\sigma_L\Vert x-x^*\Vert,\quad\forall x,x^*\in X.
\end{equation}
\end{assump}

The number $\sigma_L$ bounds the per unit distance multiplicative spread of the standard deviation of the oracle relative to the reference point $x^*\in X$. Precisely, $\sigma(x)-\sigma(x^*)\le\sigma_L$, if $\Vert x-x^*\Vert\le1$. Condition \eqref{equation:nonuniform:variance:intro} is much weaker than \eqref{equation:uniform:variance:intro} since it allows the more aggressive setting where
\begin{equation}
\sup_{x\in X}\esp\left[\Vert\nabla F(x,\xi)-\nabla f(x)\Vert^2\right]=\infty,\label{equation:unbounded:variance}
\end{equation}
and, in case $X$ is compact, it implies \eqref{equation:uniform:variance:intro} with $\sigma^2:=2\sigma(x^*)^2+2\sigma_L^2\diam(X)^2$, where $\diam(X)$ is the diameter of $X$. However, we note the quadratic dependence on $\diam(X)$ and that $\sigma^2\gg\min\{\sigma^2(x^*),\sigma_L^2\}$ if $\sqrt{\diam(X)}$ is large. In this sense, condition \eqref{equation:nonuniform:variance:intro} exploits \emph{local variance} behavior of the noise when using stochastic approximation. One of our main objectives in this work is to consider the more general condition \eqref{equation:nonuniform:variance:intro} in stochastic approximation algorithms for smooth convex optimization. We refer the reader to \cite{bach2014}, where \emph{local strong convexity} is exploited in order to ensure improved error bounds of order $\mathcal{O}(t^{-1})$ in terms of better constants.\footnote{The local strong convexity modulus around the unique solution $x^*$ is potentially much higher than the global strong convexity modulus $c$.}

The terminology ``multiplicative noise'' and generality of \eqref{equation:nonuniform:variance:intro} are explained in the following lemma whose proof is in the Appendix.
\begin{lemma}\label{lemma:multiplicative:noise:from:random lipschitz}
Suppose the random variable $\mathsf{L}:\Xi\rightarrow\re_+$ has finite variance and for a.e. $\xi\in\Xi$,
\begin{eqnarray}\label{equation:lipschitz:random:intro}
\Vert\nabla F(x,\xi)-\nabla F(y,\xi)\Vert\le\mathsf{L}(\xi)\Vert x-y\Vert,\quad x,y\in X.
\end{eqnarray}
Then \eqref{equation:smoothness:intro} holds with $L:=\Lnorm{\mathsf{L}(\xi)}$ and Assumption \ref{assump:oracle:multiplicative:noise} holds with 
$
\sigma_L=2L.
$
\end{lemma}
Hence, condition \eqref{equation:lipschitz:random:intro} defines the smoothness of $f$ \emph{and} the variance of the oracle as in Assumption \ref{assump:oracle:multiplicative:noise}. We note that \eqref{equation:lipschitz:random:intro} is a standard assumption in stochastic optimization \cite{shapiro:dent:rus2009}. In the sense of \eqref{equation:lipschitz:random:intro}, the random variable $\mathsf{L}$ is indeed a multiplicative noise on the first order SO. Under the assumptions of Lemma \ref{lemma:multiplicative:noise:from:random lipschitz}, Assumption \ref{assump:oracle:multiplicative:noise} is merely a \emph{finite second moment} assumption of $\Vert\nabla F(\xi,x^*)-\nabla f(x^*)\Vert$ for some $x^*\in X$ and $\mathsf{L}(\xi)$. We next show a typical example.

\begin{example}[Stochastic quadratic optimization with a random matrix]\label{example:1}
Consider the stochastic \emph{quadratic optimization} problem
$$
\min_{x\in\re^d}\left\{f(x):=\frac{1}{2}\left\langle x,\esp[A(\xi)]x\right\rangle+\left\langle\esp[b(\xi)],x\right\rangle\right\},
$$
where $A:\Xi\rightarrow\re^{d\times d}$ is a semi-positive definite \emph{random matrix} with nonnull mean and finite second moment and $b:\Xi\rightarrow\re^d$ is an integrable random vector. This is the case, e.g., of Least-Squares Regression \cite{dieuleveut:flammarion:bach2016,flammarion:bach2017}. After some straightforward calculation, we can derive the expression
$
\sigma(x)^2=\langle x,Bx\rangle,
$
for any $x\in \re^d$, where $B:=\sum_{i=1}^d\cov[A_i(\xi)]$, the vectors $A_1(\xi),\ldots,A_d(\xi)$ are the rows of $A(\xi)$ and $\cov[q(\xi)]$ defines the covariance matrix of the random vector $q:\Xi\rightarrow\re^d$. Let $N(B)$ denote the kernel of $B$ and $N(B)^\perp$ denote its orthogonal complement. Given $x\in\re^d$, we denote by $x_B$ the orthogonal projection of $x$ onto $N(B)^\perp$. Since $B$ is semi-positive definite, from the spectral theorem we get
$$
\sigma(x)^2\ge\lambda_+(B)\Vert x_B\Vert^2,\quad \forall x\in\re^d,
$$
where $\lambda_+(B)$ is the smallest nonnull eigenvalue of $B$. This shows \emph{$\sigma(\cdot)^2$ is quadratically coercive on the linear subspace $N(B)^\perp$}. In particular, if $B$ is positive definite, $\sigma(x)^2\ge\lambda_+(B)\Vert x\Vert^2$, for all $x\in\re^d$. We thus conclude that  \eqref{equation:uniform:variance:intro} is not valid if $X\cap N(B)^\perp$ is unbounded. In particular, it is not true in the unconstrained case ($X=\re^d$). Nevertheless, \eqref{equation:lipschitz:random:intro} holds with 
$
\mathsf{L}(\xi):=\sup\left\{\Vert A(\xi)x\Vert:\Vert x\Vert=1\right\}.
$
\end{example}

\subsection{Related work and contributions}

The performance of a SA method can be measured by its \emph{iteration} and \emph{oracle complexities} given a tolerance $\epsilon>0$ on the mean optimality gap. The first is the total number of iterations, a measure for the optimization error, while the second is the total number of samples and oracle calls, a measure for the estimation error. Statistical lower bounds \cite{agarwal:barlett:ravikumar:wainwright2012} show that the optimal oracle complexities are $\mathcal{O}(\epsilon^{-2})$ for the smooth convex class and $\mathcal{O}(\epsilon^{-1})$ for the smooth strongly convex class. Anyhow, an important question that remains is (\textsf{Q}): \emph{What is the optimal iteration complexity such that a (near) optimal oracle complexity is respected}? Related to such question is the ability of method to treat the \emph{oracle's assumptions and variance} with \emph{sampling efficiency}. 

The initial version of the SG method uses \emph{one oracle call per iteration} with stepsizes satisfying $\sum_t\alpha_t=\infty$ and $\sum_t\alpha_t^2<\infty$, typically $\alpha_t=\mathcal{O}(t^{-1})$. A fundamental improvement in respect to estimation error was Polyak-Ruppert's \emph{iterate averaging} scheme \cite{polyak1991,polyak:juditsky1992,ruppert1988,nem:rudin1978}, where \emph{longer stepsizes} $\alpha_t=\mathcal{O}(t^{-\frac{1}{2}})$ are used with a subsequent \emph{final average} of the iterates with weights $\{\alpha_t\}$ (this is sometimes called \emph{ergodic average}). \emph{If one oracle call per iteration is assumed}, such scheme obtains optimal iteration and oracle complexities of $\mathcal{O}(\epsilon^{-2})$ for the smooth convex class and of $\mathcal{O}(\epsilon^{-1})$ for smooth strongly convex class. These are also the size of the final ergodic average, a measure of the additional \emph{averaging effort} implicitly required in iterate averaging schemes. Such methods, hence, are efficient in terms of oracle complexity. Iterate averaging was then extensively explored (see e.g.  \cite{juditsky:nazin:tsybakov:vayatis2005,juditsky:rigollet:tsybakov2008,nesterov:vial2008,nesterov2009,%
nem:jud:lan:shapiro2009,xiao2010,lee:wright2012}). The important work \cite{nem:jud:lan:shapiro2009} exploits the robustness of iterate averaging in SA methods and shows that such schemes can outperform the SAA approach on relevant convex problems. On the strongly convex class, \cite{bach:moulines2011} gives a detailed non-asymptotic robust analysis of Polyak-Ruppert averaging scheme. See also  \cite{hazan:hale2014,rosasco:silvia:vu2014,needell:srebro:ward2015} for improvements.

In the seminal work \cite{nesterov1983} of Nesterov, a novel accelerated scheme for unconstrained smooth convex optimization with an exact oracle is presented obtaining the optimal rate $\mathcal{O}(Lt^{-2})$. This improves upon $\mathcal{O}(t^{-1})$ of standard methods. Motivated by the importance of regularized problems, this result was further generalized for \emph{composite} convex optimization in \cite{nesterov2007,beck:teboulle2009,tseng2008} and in \cite{lan2012}, where the stochastic oracle was considered. Assuming \eqref{equation:uniform:variance:intro} and one oracle call per iteration, \cite{lan2012} obtains iteration and oracle complexities of $\mathcal{O}\left(\sqrt{L\epsilon^{-1}}+\sigma^2\epsilon^{-2}\right)$, allowing larger values of $L$. See also \cite{hu:kwok:pan2009,xiao2010}. In \cite{ghadimi:lan2012,ghadimi:lan2013}, the strongly convex class was considered and iteration and oracle complexities of 
$
\mathcal{O}\left(\sqrt{Lc^{-1}}\ln(\epsilon^{-1})+\sigma^2c^{-1}\epsilon^{-1}\right)
$
were obtained, allowing for larger values of the condition number $L/c$. See  \cite{xiao2010,lee:wright2012,lin:chen:pena2014} for considerations on sparse solutions.

Motivated by question (\textsf{Q}), a rapidly growing line of research is the development of methods which use \emph{more than one oracle call per iteration}. Current examples are the \emph{gradient aggregation} and \emph{dynamic sampling} methods (see \cite{bottou:curtis:nocedal2016}, Section 5). Designed for \emph{finitely supported distributions} of the form \eqref{problem:SAA} (as found in the ERM problem of machine learning), gradient aggregation methods reduce the variance by combining in a specific manner eventual exact computation (or storage) of gradients and eventual iterate averaging (or randomization schemes) with frequent gradient sampling. Their complexity bounds typically grow with the sample size $N$ in \eqref{problem:SAA}. See e.g. \cite{leroux:schmidt:bach2012,schmidt:leroux:bach2013,defazio:bach:lacoste-julien2014,%
johnson:zhang2013,xiao:zhang2014,gurbuzbalaban:ozdaglar:parrilo2015,lin:mairal:harchaoui2015,%
woodworth:srebro2016,allen-zhu2016} and references therein. Designed for solving problem \eqref{equation:expected:valued:objective}-\eqref{problem:min:intro} with an online data acquisition (as is the case in many stochastic and simulation optimization problems based on Monte Carlo methods), dynamic sampling methods reduce variance by estimating the gradient via an \emph{empirical average} associated to a sample whose size (\emph{mini-batch}) is increased at every iteration. Their complexity bounds typically grow with the oracle variance $\sigma^2$ in \eqref{equation:uniform:variance:intro}. See \cite{byrd:chiny:nocedal:wu2012,%
friedlander:schmidt2013,zhang:yang:jin:he2013,ghadimi:lan2016,iusem:jofre:oliveira:thompson2017,%
iusem:jofre:oliveira:thompson2016} and references therein. An essential point related to (\textsf{Q}) is if such increased effort in computation per iteration used is worth. A nice fact is that current gradient aggregation and dynamic sampling methods achieve the order of the deterministic optimal iteration complexity with the same (near) optimal oracle complexity and averaging effort of standard iterate averaging schemes. \emph{In this sense}, gradient aggregation and dynamic sampling methods are thus a more efficient option.

The contributions of this work may be summarized as follows:
\begin{quote}
\emph{We show the standard global assumption \eqref{equation:uniform:variance:intro} used to obtain non-asymptotic convergence can be replaced by the significantly weaker local assumption \eqref{equation:nonuniform:variance:intro}. The bounds derived under \eqref{equation:nonuniform:variance:intro} and dynamic sampling depend on the local variances $\sigma(x^*)^2$ and $\sigma_L^2$ for $x^*\in\solset(f,\varphi)$. Moreover, such bounds can be tuned so to resemble previous bounds obtained if it was supposed that \eqref{equation:uniform:variance:intro} holds but replacing $\sigma^2$ with $\sigma(x^*)^2$. In particular, if \eqref{equation:uniform:variance:intro} does hold then our bounds are significantly sharper then previous ones since typically $\sigma(x^*)^2\ll\sigma^2$. These type of results can be seen as a \emph{variance localization property}.} 
\end{quote}
We next state our contributions more precisely.

(i) \textsf{Stochastic smooth non-strongly convex optimization}: for this class of problems, we propose \textsf{Algorithm \ref{algorithm:smooth:convex}} stated in Section \ref{section:smooth:convex} which is a dynamic sampled SA version of the accelerated method FISTA \cite{beck:teboulle2009}. We show \textsf{Algorithm \ref{algorithm:smooth:convex}} achieves the optimal iteration complexity of $\mathcal{O}(\epsilon^{-\frac{1}{2}})$ with a near-optimal oracle complexity and average effort of $\mathcal{O}(\epsilon^{-2})[\ln(\epsilon^{-\frac{1}{2}})]^2$ under the more general Assumption \ref{assump:oracle:multiplicative:noise} of \emph{multiplicative noise}. These are online bounds.\footnote{That is, without an a priori known number of iterations.} The factor $[\ln(\epsilon^{-\frac{1}{2}})]^2$ can be removed for offline bounds. This improves upon \cite{ghadimi:lan2016}, where accelerated dynamic sampling schemes obtain the same complexities but with the more restrictive assumption \eqref{equation:uniform:variance:intro} of an oracle with \emph{uniformly bounded variance}. Hence, we do not implicitly require an additive noise model nor boundedness of $X$. Our rates depend on the local variances $\sigma(x^*)^2$ for $x^*\in\mbox{S}(f,\varphi)$ and $\sigma_L^2$. Interestingly, the stepsize and sampling rate policies can be tuned so that our bounds resemble those obtained if \emph{it was supposed \eqref{equation:uniform:variance:intro} holds} but with $\sigma^2$ replaced by $\sigma(x^*)^2$ for some $x^*\in\mbox{S}(f,\varphi)$. Hence, in case \eqref{equation:uniform:variance:intro} indeed holds, our bounds are much sharper than \cite{ghadimi:lan2016} since typically $\sigma(x^*)^2\ll\sigma^2$. See Example \ref{example:1} in Section \ref{subsection:oracle:assumptions} and Theorem \ref{thm:rate:smooth:convex:multiplicative:noise}, Corollary \ref{cor:rate:complexity:smooth:convex:multiplicative:noise} and Remark \ref{remark:constants:smooth:convex} in Section \ref{section:smooth:convex}. Additionally, \textsf{Algorithm \ref{algorithm:smooth:convex}} is sparsity preserving since it is based on FISTA. Differently, the methods in \cite{ghadimi:lan2016} are not sparsity preserving since they are based on the AC-SA method of \cite{lan2012} which uses iterate averaging (see  \cite{xiao2010,lee:wright2012,lin:chen:pena2014} for further observations). Finally, we do not use randomization procedures as in \cite{ghadimi:lan2016} which require additional sampling effort of an auxiliary random variable. 

(ii) \textsf{Stochastic smooth strongly convex optimization}: for this class of problems we propose \textsf{Algorithm \ref{algorithm:smooth:strongly:convex}} given in Section \ref{section:smooth:strongly:convex} which is a dynamic sampled version of the stochastic proximal gradient method. We show \textsf{Algorithm \ref{algorithm:smooth:strongly:convex}} achieves the optimal iteration complexity of $\mathcal{O}(\ln(\epsilon^{-1}))$ with an optimal oracle complexity and average effort of $\mathcal{O}(\epsilon^{-1})$ under the more general Assumption \ref{assump:oracle:multiplicative:noise}. This improves upon \cite{byrd:chiny:nocedal:wu2012} where the same complexities are obtained for dynamic sampling schemes but with the more restrictive assumption \eqref{equation:uniform:variance:intro}. Also, no regularization nor constraints are addressed in \cite{byrd:chiny:nocedal:wu2012} (i.e. $\varphi\equiv0$ and $X=\re^d$). Again, a consequence of our results is that no boundedness or additive noise assumptions are required. In terms of variance our bounds are local: they depend on $\sigma(x^*)^2$ and $\mathsf{Q}(x^*,t_0):=\sigma_L^2\max_{1\le t\le t_0-1}\esp[\Vert x^\tau-x^*\Vert^2]$ for some small number $t_0\sim\ln(\kappa/N_0)$, where $N_0$ is the initial batch size, $\kappa:=\frac{L}{c}$ is the condition number and $x^*$ is the unique solution. Interestingly, if an upper bound on $\kappa$ is known, the stepsize and sampling rate policies can be tuned so that our bounds resemble those obtained if \emph{it was supposed \eqref{equation:uniform:variance:intro} holds} but with $\sigma^2$ replaced by $\sigma(x^*)^2$ (without dependence on $\mathsf{Q}(t_0,x^*)$). We thus conclude a \emph{regularization property} for strongly convex functions: the more \emph{well-conditioned} the optimization problem is or the more aggressive the \emph{initial batch size} is, the more our error bounds approach those obtained if it was supposed \eqref{equation:uniform:variance:intro} holds but with $\sigma^2$ replaced by $\sigma(x^*)^2$ (with this bound exactly achieved if an upper bound on $\kappa$ is known). In case \eqref{equation:uniform:variance:intro} indeed holds, our bounds are thus much sharper than \cite{byrd:chiny:nocedal:wu2012} since typically $\sigma(x^*)^2\ll\sigma^2$ and\footnote{We remark that $\max_{1\le\tau\le T}\esp[\Vert x^\tau-x^*\Vert^2]\ll\diam(X)^2$ always holds after a \emph{finite} number $T$ of iterations.} $\mathsf{Q}(t_0,x^*)\ll\sigma^2$ if  $\max_{1\le t\le t_0-1}\esp[\Vert x^\tau-x^*\Vert^2]\ll\diam(X)^2$. See Example \ref{example:1} in Section \ref{subsection:oracle:assumptions} and Theorem \ref{thm:rate:strong:convex}, Corollary \ref{cor:rate:complexity:strong:convex} and Remark \ref{remark:constants:strongly:convex} in Section \ref{section:smooth:strongly:convex}.

Let's call a SA method  (near) optimal if it achieves the order of the deterministic optimal iteration complexity with (near) optimal oracle complexity and average cost. To the best of our knowledge, we give for the first time (near) optimal SA methods for stochastic smooth CO with an oracle with \emph{multiplicative noise} and a \emph{general heavy-tailed distribution} (with finite 2nd moment). This is an important improvement since, in principle, it is not obvious that SA methods can converge if the oracle has an unbounded variance satisfying \eqref{equation:unbounded:variance}. This is even less obvious when using acceleration a la FISTA since, in that case, an extrapolation is computed with divergent weights $\beta_t=\mathcal{O}(t)$ (see Theorem \ref{thm:rate:smooth:convex:multiplicative:noise} and Remark \ref{remark:constants:smooth:convex}). Also, the introduction of a regularization term $\varphi$ and constraints is nontrivial in the setting of a multiplicative noise. The main reason for obtaining non-asymptotic convergence under Assumption \ref{assump:oracle:multiplicative:noise} and unbounded gradients is the use of dynamic sampling with a careful sampling rate, i.e., one that uses the same oracle complexity of the standard iterate averaging scheme. In these methods typically \emph{one ergodic average of iterates} with size $T$ at the \emph{final $T$-th iteration} is used. Differently, in dynamic sampling schemes \emph{local empirical averages of gradients} with smaller and increasing sizes are distributed \emph{along iterations}.\footnote{The possibility of distributing the empirical averages along iterations is possible due to the on-line nature of the SA method. This is not shared by the SAA methodology which is an off-line procedure.} This is also the reason our bounds depend on local variances at points of the trajectory of the method and at points of S$(f,\varphi)$ (but not on the whole $X$). Such results are not shared by the SA method with constant $N_k$ for ill-conditioned problems nor for the SAA method.

We review some works besides \cite{ghadimi:lan2016,byrd:chiny:nocedal:wu2012}. A variation of Assumption \ref{assump:oracle:multiplicative:noise} was proposed by Iusem, Jofr\'e, Oliveira and Thompson in \cite{iusem:jofre:oliveira:thompson2017}, but their method is tailored at solving monotone variational inequalities with the extragradient method. Hence, on the class of smooth convex functions, the suboptimal iteration complexity of $\mathcal{O}(\epsilon^{-1})$ is achieved with the use of an additional proximal step (not required for optimization). In  \cite{dieuleveut:flammarion:bach2016,flammarion:bach2017}, the assumption of multiplicative noise is also analyzed but their rate analysis is for the special class of stochastic convex quadratic objectives as in Example \ref{example:1}, where the main motivation is to study linear LSR problems. They obtain offline iteration and oracle complexities of $\mathcal{O}(\epsilon^{-1})$ on the class of quadratic convex functions using Tykhonov regularization with iterate averaging. Differently, our analysis is optimal on the general class of smooth convex functions. See also the recent works \cite{jain:kakade:kidambi:netrapalli:sidford2016,jain:kakade:kidambi:netrapalli:sidford2017} on LSR problems.

We focus now on the class of smooth strongly convex functions. In \cite{friedlander:schmidt2013} the dynamic sampled SG method is also analyzed. However, their analysis strongly relies on finitely supported distributions as in aggregated methods, no oracle complexity bounds are provided and it is assumed no regularization nor constraints ($\varphi\equiv0$ and $X=\re^d$). In \cite{zhang:yang:jin:he2013}, a method using $\ln(\epsilon^{-1})$ projections for smooth strongly CO is proposed but still assuming boundedness of the oracle and obtaining the suboptimal iteration complexity $\mathcal{O}(\epsilon^{-1})$. The works  \cite{bach:moulines2011,rosasco:silvia:vu2014,needell:srebro:ward2015} do not require the assumption \eqref{equation:uniform:variance:intro} and cope with multiplicative noise. However, their iteration and oracle complexities are $\mathcal{O}(\epsilon^{-1})$ and, hence, suboptimal when compared to our method.

Finally, we compare the results of item (ii) with \cite{frostig:ge:kakade:sidford2015}, where also dynamic sampled optimal methods for stochastic smooth strongly convex optimization are derived (with knowledge of $\kappa$). The class of smooth functions analyzed in \cite{frostig:ge:kakade:sidford2015} are smaller, requiring them to be \emph{twice differentiable}. With respect to statistical assumptions, their analysis allow multiplicative noise but it uses a condition much stronger than Assumption \ref{assump:oracle:multiplicative:noise}. Indeed, they assume a.e. $\nabla F(\cdot,\xi)$ is $\mathsf{L}(\xi)$-Lipschitz continuous with a \emph{bounded Lipschitz modulus}, that is, 
$
\sup_{\xi\in\Xi}\mathsf{L}(\xi)<\infty.
$
As a consequence, $\mathsf{L}(\cdot)-\esp[\mathsf{L}(\xi)]$ is a sub-Gaussian random variable (whose tail decreases exponentially fast).\footnote{We say a centered random variable $q:\Xi\rightarrow\re$ is sub-Gaussian with parameter $\sigma^2$ if $\esp\left\{e^{uq(\xi)}\right\}\le e^{\frac{\sigma^2u^2}{2}}$ for all $u\in\re$. We note that a centered Gaussian variable $N$ with variance $\sigma^2$ satisfies $\esp\left\{e^{uN(\xi)}\right\}=e^{\frac{\sigma^2u^2}{2}}$ for all $u\in\re$.} From \eqref{equation:lipschitz:random:intro}, this implies that $\nabla F(x,\cdot)$ is also sub-Gaussian. In our Assumption \ref{assump:oracle:multiplicative:noise}, we allow heavy-tailed distributions (with finite 2nd moment). Precisely, for any $x\in X$, we only require that $\esp[\Vert\nabla F(x,\cdot)\Vert^2]<\infty$ so that the fluctuations of $\nabla F(x,\cdot)$ can be much heavier than a Gaussian random variable. The work in \cite{frostig:ge:kakade:sidford2015} does not consider regularization nor constraints (i.e., $\varphi\equiv0$ and $X=\re^d$), ignoring effects of $\varphi$ and $X$ on the oracle's variance.\footnote{In the case that $X=\re^d$ and $\varphi\equiv0$, the unique solution $x^*$ satisfies $\nabla f(x^*)=0$ so that $\sigma(x^*)^2=\esp[\Vert\nabla F(x^*,\xi)\Vert^2]$.} Finally, their methods differ significantly from ours. Indeed, the methods in \cite{frostig:ge:kakade:sidford2015} are SA versions of the SVRG method of \cite{johnson:zhang2013,xiao:zhang2014} originally designed for finite-sum objectives. Hence, besides dynamic mini-batching they require in every iteration an inner loop of $m$ iterations to further reduce the variance of the gradients. This implies the need of a randomization scheme and, at least, an additional $m\ge 400\cdot3^6\kappa$ number of samples per iteration (which can be large if $\kappa\gg1$. See \cite{frostig:ge:kakade:sidford2015}, Corollary 4). Our method is solely based on the simple stochastic gradient method and we only use dynamic mini-batching to reduce variance with no additional randomization and sampling.

In Section \ref{section:preliminaries:notation} preliminaries and notation are presented while in Sections \ref{section:smooth:convex}-\ref{section:smooth:strongly:convex} our convergence theory is presented for non-strongly and strongly convex problems respectively. Technical results are proved in the Appendix.

\section{Preliminaries and notation}\label{section:preliminaries:notation}

For $x,y\in\re^n$, we denote $\langle x,y\rangle$ the standard inner product, and $\Vert x\Vert=\sqrt{\langle x,x\rangle}$ the correspondent Euclidean norm. Given $C\subset\re^n$ and $x\in\re^n$, we use the notation $\dist(x,C):=\inf\{\Vert x-y\Vert:y\in C\}$. Given sequences $\{x^k\}$ and $\{y^k\}$, we use the notation $x^k=\mathcal{O}(y^k)$ or $\Vert x^k\Vert\lesssim\Vert y^k\Vert$ to mean that there exists a constant $C>0$ 
such that $\Vert x^k\Vert\le C\Vert y^k\Vert$ for all $k$. The notation 
$\Vert x^k\Vert\sim\Vert y^k\Vert$ means that $\Vert x^k\Vert\lesssim\Vert y^k\Vert$ and $\Vert y^k\Vert\lesssim\Vert x^k\Vert$. Given a $\sigma$-algebra $\alg$ and a random variable $\xi$, we denote by $\esp[\xi]$ and $\esp[\xi|\alg]$ the expectation and conditional expectation, respectively. We write $\xi\in\alg$ for ``$\xi$ is $\alg$-measurable'' and $\xi\perp\perp\alg$ for ``$\xi$ is independent of $\alg$''. We denote by $\sigma(\xi_1,\ldots,\xi_k)$ the $\sigma$-algebra generated by the random variables $\xi_1,\ldots,\xi_k$. Given the random variable $\xi$ and $p\ge1$, $\Lpnorm{\xi}$ is the $L^p$-norm of $\xi$ and 
$
\Lpnorm{\xi\,|\alg}:=\sqrt[p]{\esp\left[|\xi|^p\,|\alg\right]}
$
is the $L_p$-norm of $\xi$ conditional to the $\sigma$-algebra $\alg$. By $\lfloor x\rfloor$ and $\lceil x\rceil$ we mean the lowest integer greater and the highest integer lower than $x\in\re$, respectively. We use the notation $[m]:=\{1,\ldots,m\}$ for $m\in\mathbb{N}$ and $\mathbb{N}_0:=\mathbb{N}\cup\{0\}$. Given $a,b\in\re$, $a\vee b:=\max\{a,b\}$.

Given a function $p:X\rightarrow\re$, $y\in X$ and $v\in\re^d$, we define
\begin{eqnarray}
z\mapsto\ell_p(y,v;z)&:=&p(y)+\langle v,z-y\rangle,\label{equation:linearization:directed}
\end{eqnarray}
i.e., the linearization of $p$ at the point $y$ and direction $v$. If moreover 
$p$ is differentiable, we define
\begin{eqnarray}
z\mapsto\ell_p(y;z):=p(y)+\langle \nabla p(y),z-y\rangle,\label{equation:linearization}
\end{eqnarray}
i.e., the linearization of $p$ at $y$. The \emph{prox-mapping} with respect to $X$ and a given convex function $\varphi:X\rightarrow\re$ is defined as 
$
P_y^\varphi[u]:=\argmin_{x\in X}\left\{\langle u,x-y\rangle+\frac{1}{2}\Vert x-y\Vert^2+\varphi(x)\right\}.
$
The following two properties are well known.
\begin{lemma}\label{lemma:proximal inequality}
Let $p:X\rightarrow\re$ be a convex function and $\alpha>0$. For any $y\in X$, if
$
z\in\argmin_{X}\left\{p+\frac{1}{2\alpha}\Vert \cdot-y\Vert^2\right\},
$
then
$$
p(z)+\frac{1}{2\alpha}\Vert z-y\Vert^2\le p(x)+\frac{1}{2\alpha}\Vert x-y\Vert^2-\frac{1}{2\alpha}\Vert x-z\Vert^2,\quad\forall x\in X.
$$
\end{lemma}
\begin{proof}
See Lemma 1 in \cite{lan2012}.\qed
\end{proof}

\begin{lemma}\label{lemma:inner:upper:approximation:inequality}
Let $p:X\rightarrow\re$ be a smooth convex function with $L$-Lipschitz gradient. Set $c_p:=c>0$ if $p$ is $c$-strongly convex and $c_p=0$ otherwise. Then the following relations hold:
\begin{eqnarray*}
\ell_p(y;x)+\frac{c_p}{2}\Vert x-y\Vert^2\le p(x)\le \ell_p(y;x)+\frac{L}{2}\Vert x-y\Vert^2,\forall x,y\in X.
\end{eqnarray*}
\end{lemma}
\begin{proof}
See Lemma 2 in \cite{lan2012}.\qed
\end{proof}

For smooth $c$-strongly convex functions $p:X\rightarrow\re$, the following fundamental bounds hold:
\begin{eqnarray}
\frac{c}{2}\Vert x-x^*\Vert^2\le p(x)-p(x^*)\le\frac{1}{2c}\Vert\nabla p(x)\Vert^2,\forall x\in X,\label{equation:PL:inequality:strongly:convex}
\end{eqnarray}
where $x^*$ is the unique solution of $\min_{x\in\re^d}p(x)$. The first inequality above is an immediate consequence of Lemma \ref{lemma:inner:upper:approximation:inequality} and the first order necessary optimality condition. For the second inequality see, e.g., \cite{bottou:curtis:nocedal2016}, relation (4.12). We will use the following lemma many times. The proof is left to the Appendix.
\begin{lemma}[Oracle's variance decay under multiplicative noise]\label{lemma:oracle:variance:decay}
Suppose \eqref{equation:expected:valued:objective} and Assumption \ref{assump:oracle:multiplicative:noise} hold. Given an i.i.d. sample $\{\xi_j\}_{j=1}^N$ drawn from the distribution $\probn$ and $x\in\re^d$, set
$$
\epsilon(x):=\sum_{j=1}^N\frac{\nabla F(x,\xi_j)-\nabla f(x)}{N}.
$$
Then
$
\Lnorm{\Vert\epsilon(x)\Vert}\le\frac{\sigma(x^*)+\sigma_L\Vert x-x^*\Vert}{\sqrt{N}}.
$
\end{lemma}

\section{Smooth convex optimization with multiplicative noise and acceleration}\label{section:smooth:convex}

We propose \textsf{Algorithm \ref{algorithm:smooth:convex}} for the composite problem \eqref{problem:min:composite} assuming an stochastic oracle satisfying \eqref{equation:expected:valued:objective} and Assumption \ref{assump:oracle:multiplicative:noise} of a multiplicative noise in the case the objective $f$ is smooth convex satisfying \eqref{equation:smoothness:intro}.
\begin{algorithm}
\caption{Stochastic approximated FISTA with dynamic mini-batching}\label{algorithm:smooth:convex}
\begin{algorithmic}[1]
  \scriptsize
  \STATE INITIALIZATION: initial iterates $y^1=z^0$, positive stepsize sequence $\{\alpha_t\}$, positive weights $\{\beta_t\}$ and sampling rate $\{N_t\}$.
  \STATE ITERATIVE STEP: Given $y^t$ and $z^{t-1}$, generate i.i.d. sample $\xi^t:=\{\xi_j^t\}_{j=1}^{N_t}$ from $\probn$ independently from previous samples. Compute
  $$ 
  F_t'(y^t,\xi^t):=\frac{1}{N_t}\sum_{j=1}^{N_t}\nabla F(y^t,\xi^t_j).
  $$
  Then set
  	\begin{eqnarray}
	z^t&:=&P^{\alpha_t\varphi}_{y^t}\left[\alpha_t F'(y^t,\xi^t)\right]=
	\argmin_{x\in X}\left\{\ell_f\left(y^{t},F'(y^{t},\xi^{t});x\right)+\frac{1}{2\alpha_{t}}\Vert x-y^{t}\Vert^2+\varphi(x)\right\},\label{equation:fista:eq1}\\
	y^{t+1}&:=&\frac{(\beta_t-1)}{\beta_{t+1}}(z^t-z^{t-1})+z^t,\label{equation:fist:eq2}
	\end{eqnarray}
\end{algorithmic}
\end{algorithm}
For $t\in\mathbb{N}$, we define the oracle error
$
\epsilon^t := F_t'(y^t,\xi^t)-\nabla f(y^t)
$
and the filtration
$
\alg_t:=\sigma(y^1,\xi^1\ldots,\xi^t).
$

\subsection{Derivation of an error bound}

In the following, it will be convenient to define, for any $t\ge2$,
\begin{eqnarray}
s^{t}:=\beta_t z^t-(\beta_t-1)z^{t-1}.\label{equation:def:sk}
\end{eqnarray}

\begin{proposition}
\label{proposition:error:bound:fista}
Suppose Assumptions \ref{assump:iid:sampling}-\ref{assump:oracle:multiplicative:noise} hold for the problem \eqref{problem:min:composite} satisfying \eqref{equation:expected:valued:objective} and \eqref{equation:smoothness:intro}. Suppose $\{\alpha_t\}$ is non-increasing and
$$
\alpha_t\in\left(0,\frac{1}{L}\right),\quad\beta_t\ge1,\quad\beta_t^2=\beta_{t+1}^2-\beta_{t+1},\quad\forall t\in\mathbb{N}.
$$
Then the sequence generated by \textsf{\emph{Algorithm \ref{algorithm:smooth:convex}}} satisfies: for all $t\in\mathbb{N}$ and $x^*\in\mbox{\emph{S}}(f,\varphi)$, 
\begin{eqnarray*}
2\alpha_{t+1}\beta_{t+1}^2\left[g(z^{t+1})-g^*\right]-2\alpha_t\beta_t^2\left[g(z^t)-g^*\right] &\le &
\Vert s^t-x^*\Vert^2-\Vert s^{t+1}-x^*\Vert^2
+\frac{\alpha_{t+1}\beta_{t+1}^2}{(\alpha_{t+1}^{-1}-L)}\Vert\epsilon^{t+1}\Vert^2\nonumber\\
&&+2\alpha_{t+1}\beta_{t+1}\langle\epsilon^{t+1},x^*-s^t\rangle.
\end{eqnarray*}
\end{proposition}
\begin{proof}
For convenience of notation, we will use the notation
$
v_t:=g(z^t)-g^*.
$
We start by deriving the following fundamental inequality. We have, for any $t\in\mathbb{N}_0$ and $x\in X$,
\begin{eqnarray}
g(z^{t+1})&\le &\ell_f(y^{t+1};z^{t+1})+\varphi(z^{t+1})+\frac{L}{2}\Vert z^{t+1}-y^{t+1}\Vert^2\nonumber\\
&=&\ell_f\left(y^{t+1},F'(y^{t+1},\xi^{t+1});z^{t+1}\right)+\varphi(z^{t+1})+\frac{L}{2}\Vert z^{t+1}-y^{t+1}\Vert^2
-\langle\epsilon^{t+1},z^{t+1}-y^{t+1}\rangle\nonumber\\
&\le &\ell_f\left(y^{t+1},F'(y^{t+1},\xi^{t+1});x\right)+\varphi(x)+\frac{\alpha_{t+1}^{-1}}{2}\Vert x-y^{t+1}\Vert^2-\frac{\alpha_{t+1}^{-1}}{2}\Vert x-z^{t+1}\Vert^2+\frac{(L-\alpha_{t+1}^{-1})}{2}\Vert z^{t+1}-y^{t+1}\Vert^2\nonumber\\
&&-\langle\epsilon^{t+1},z^{t+1}-y^{t+1}\rangle\nonumber\\
&=&\ell_f(y^{t+1};x)+\varphi(x)+\frac{\alpha_{t+1}^{-1}}{2}\Vert x-y^{t+1}\Vert^2-\frac{\alpha_{t+1}^{-1}}{2}\Vert x-z^{t+1}\Vert^2+\frac{(L-\alpha_{t+1}^{-1})}{2}\Vert z^{t+1}-y^{t+1}\Vert^2\nonumber\\
&&+\langle\epsilon^{t+1},x-y^{t+1}\rangle-\langle\epsilon^{t+1},z^{t+1}-y^{t+1}\rangle\nonumber\\
&\le &g(x)+\frac{\alpha_{t+1}^{-1}}{2}\Vert x-y^{t+1}\Vert^2-\frac{\alpha_{t+1}^{-1}}{2}\Vert x-z^{t+1}\Vert^2+
\frac{(L-\alpha_{t+1}^{-1})}{2}\Vert z^{t+1}-y^{t+1}\Vert^2\nonumber+\langle\epsilon^{t+1},x-z^{t+1}\rangle,\nonumber\\\label{equation:prox:and:approx:inequalities:SA}
\end{eqnarray}
where we used $g(z^{t+1})=f(z^{t+1})+\varphi(z^{t+1})$ and the upper inequality of Lemma \ref{lemma:inner:upper:approximation:inequality} with $p:=f$ in first inequality (by smoothness of $f$), definitions \eqref{equation:linearization:directed}-\eqref{equation:linearization} and $\nabla f(y^{t+1})=F'(y^{t+1},\xi^{t+1})-\epsilon^{t+1}$ in the first equality, the expression \eqref{equation:fista:eq1}
and Lemma \ref{lemma:proximal inequality} with the convex function
$
p:=\ell_f(y^{t+1},F'(y^{t+1},\xi^{t+1});\cdot)+\varphi,
$
$\alpha:=\alpha_{t+1}$, $y:=y^{t+1}$ and $z:=z^{t+1}$ in second inequality, definitions \eqref{equation:linearization:directed}-\eqref{equation:linearization} and $F'(y^{t+1},\xi^{t+1})=\nabla f(y^{t+1})+\epsilon^{t+1}$ in the second equality as well as $g(x)=f(x)+\varphi(x)$ and Lemma \ref{lemma:inner:upper:approximation:inequality} with $p:=f$ in the last inequality (by convexity of $f$).

We now set $x:=z^t$ and $x:=x^*\in\mbox{S}(f,\varphi)$ in \eqref{equation:prox:and:approx:inequalities:SA} obtaining
\begin{eqnarray}
[g(z^{t+1})-g^*]-[g(z^t)-g^*]&=&g(z^{t+1})-g(z^t)\nonumber\\
&\le &\frac{\alpha_{t+1}^{-1}}{2}\Vert z^t-y^{t+1}\Vert^2-\frac{\alpha_{t+1}^{-1}}{2}\Vert z^t-z^{t+1}\Vert^2+
\frac{(L-\alpha_{t+1}^{-1})}{2}\Vert z^{t+1}-y^{t+1}\Vert^2\nonumber\\
&&+\langle\epsilon^{t+1},z^t-z^{t+1}\rangle,\label{equation:error:bound:fista:eq1}\\
g(z^{t+1})-g^*&\le &\frac{\alpha_{t+1}^{-1}}{2}\Vert x^*-y^{t+1}\Vert^2-\frac{\alpha_{t+1}^{-1}}{2}\Vert x^*-z^{t+1}\Vert^2+
\frac{(L-\alpha_{t+1}^{-1})}{2}\Vert z^{t+1}-y^{t+1}\Vert^2\nonumber\\
&&+\langle\epsilon^{t+1},x^*-z^{t+1}\rangle.\label{equation:error:bound:fista:eq2}
\end{eqnarray}

We multiply by $\beta_{t+1}-1$ relation \eqref{equation:error:bound:fista:eq1}, add the result to \eqref{equation:error:bound:fista:eq2} and then further multiply by $2\alpha_{t+1}$ obtaining
\begin{eqnarray}
2\alpha_{t+1}\beta_{t+1}v_{t+1}-2\alpha_{t+1}(\beta_{t+1}-1)v_t&\le &
(\beta_{t+1}-1)\Vert z^t-y^{t+1}\Vert^2-(\beta_{t+1}-1)\Vert z^t-z^{t+1}\Vert^2\nonumber\\
&&+\Vert x^*-y^{t+1}\Vert^2-\Vert x^*-z^{t+1}\Vert^2\nonumber\\
&&+\alpha_{t+1}(L-\alpha_{t+1}^{-1})\beta_{t+1}\Vert z^{t+1}-y^{t+1}\Vert^2\nonumber\\
&&+2\alpha_{t+1}\langle\epsilon^{t+1},(\beta_{t+1}-1)z^t+x^* -\beta_{t+1}z^{t+1}\rangle\nonumber\\
&=&\beta_{t+1}\Vert z^t-y^{t+1}\Vert^2-\beta_{t+1}\Vert z^t-z^{t+1}\Vert^2\nonumber\\
&&+\Vert z^t-z^{t+1}\Vert^2-\Vert z^t-y^{t+1}\Vert^2\nonumber\\
&&+\Vert x^*-y^{t+1}\Vert^2-\Vert x^*-z^{t+1}\Vert^2\nonumber\\
&&+\alpha_{t+1}(L-\alpha_{t+1}^{-1})\beta_{t+1}\Vert z^{t+1}-y^{t+1}\Vert^2\nonumber\\
&&+2\alpha_{t+1}\langle\epsilon^{t+1},(\beta_{t+1}-1)z^t+x^* -\beta_{t+1}z^{t+1}\rangle.\label{equation:error:bound:fista:eq3}
\end{eqnarray}

We will use repeatedly the following Pythagorean relation:
\begin{equation}\label{equation:pythagoras:SA}
\Vert a-c\Vert^2-\Vert b-c\Vert^2=-\Vert b-a\Vert^2+2\langle b-a,c-a\rangle.
\end{equation}
Corresponding to the first line of \eqref{equation:error:bound:fista:eq3}, 
we multiply $\beta_{t+1}$ in \eqref{equation:pythagoras:SA} with $a:=y^{t+1}$, $b:=z^{t+1}$ and $c:=z^t$, obtaining
\begin{eqnarray}\label{equation:line1:SA}
\beta_{t+1}\Vert y^{t+1}-z^t\Vert^2-\beta_{t+1}\Vert z^{t+1}-z^t\Vert^2&=&
-\beta_{t+1}\Vert z^{t+1}-y^{t+1}\Vert^2
+2\beta_{t+1}\langle z^{t+1}-y^{t+1},z^t-y^{t+1}\rangle.
\end{eqnarray}
Corresponding to the second line of \eqref{equation:error:bound:fista:eq3} we get
\begin{eqnarray}\label{equation:line2:SA}
\Vert z^{t+1}-z^t\Vert^2-\Vert y^{t+1}-z^t\Vert^2 &=&\Vert z^{t+1}-y^{t+1}\Vert^2
+2\langle z^{t+1}-y^{t+1},y^{t+1}-z^t\rangle,
\end{eqnarray}
by multiplying \eqref{equation:line1:SA} with $-\beta_{t+1}^{-1}$. Finally, corresponding to the third line of \eqref{equation:error:bound:fista:eq3} we get
\begin{eqnarray}\label{equation:line3:SA}
\Vert y^{t+1}-x^*\Vert^2-\Vert z^{t+1}-x^*\Vert^2&=&-\Vert z^{t+1}-y^{t+1}\Vert^2
+2\langle z^{t+1}-y^{t+1},x^*-y^{t+1}\rangle,
\end{eqnarray}
using \eqref{equation:pythagoras:SA} with $a:=y^{t+1}$, $b:=z^{t+1}$ and $c:=x^*$.

Now we sum the identities \eqref{equation:line1:SA}-\eqref{equation:line3:SA} and use the result in the right hand side of \eqref{equation:error:bound:fista:eq3} obtaining
\begin{eqnarray*}
2\alpha_{t+1}\beta_{t+1}v_{k+1}-2\alpha_{t+1}(\beta_{t+1}-1)v_t&\le &
-\beta_{t+1}\Vert z^{t+1}-y^{t+1}\Vert^2+\\
&&+2\langle z^{t+1}-y^{t+1},\beta_{t+1}(z^t-y^{t+1})+x^*-z^t\rangle\\
&&+\alpha_{t+1}(L-\alpha_{t+1}^{-1})\beta_{t+1}\Vert z^{t+1}-y^{t+1}\Vert^2\\
&&+2\alpha_{t+1}\langle\epsilon^{t+1},(\beta_{t+1}-1)z^t+x^*-\beta_{t+1}z^{t+1}\rangle\\
&=&-\beta_{t+1}\Vert z^{t+1}-y^{t+1}\Vert^2+\\
&&+2\langle z^{t+1}-y^{t+1},(\beta_{t+1}-1)z^t+x^*-\beta_{t+1}y^{t+1}\rangle\\	
&&+\alpha_{t+1}(L-\alpha_{t+1}^{-1})\beta_{t+1}\Vert z^{t+1}-y^{t+1}\Vert^2\\
&&+2\alpha_{t+1}\langle\epsilon^{t+1},(\beta_{t+1}-1)z^t+x^*-\beta_{t+1}z^{t+1}\rangle.
\end{eqnarray*}

We multiply the above inequality by $\beta_{t+1}$ and use $\beta_t^2=\beta_{t+1}^2-\beta_{t+1}$ as assumed in the proposition. We then obtain
\begin{eqnarray*}
2\alpha_{t+1}\beta_{t+1}^2v_{t+1}-2\alpha_{t+1}\beta_t^2v_t &\le &
-\Vert \beta_{t+1}z^{t+1}-\beta_{t+1}y^{t+1}\Vert^2\\
&&+2\langle\beta_{t+1}z^{t+1}-\beta_{t+1}y^{t+1},(\beta_{t+1}-1)z^t+x^*-\beta_{t+1}y^{t+1}\rangle\\
&&+\alpha_{t+1}(L-\alpha_{t+1}^{-1})\beta_{t+1}^2\Vert z^{t+1}-y^{t+1}\Vert^2\\
&&+2\alpha_{t+1}\beta_{t+1}\langle\epsilon^{t+1},(\beta_{t+1}-1)z^t+x^*-\beta_{t+1}z^{t+1}\rangle.
\end{eqnarray*}

Corresponding to the first two lines in the right hand side of the previous inequality, we invoke again the Pythagorean relation \eqref{equation:pythagoras:SA} with 
$a:=\beta_{t+1}y^{t+1}$, $b:=\beta_{t+1}z^{t+1}$ and $c:=(\beta_{t+1}-1)z^t+x^*$,
obtaining
\begin{eqnarray}
2\alpha_{t+1}\beta_{t+1}^2v_{t+1}-2\alpha_{t+1}\beta_t^2v_t &\le &
\Vert\beta_{t+1}y^{t+1}-(\beta_{t+1}-1)z^t-x^*\Vert^2-\Vert \beta_{t+1}z^{t+1}-(\beta_{t+1}-1)z^t-x^*\Vert^2\nonumber\\
&&+\alpha_{t+1}(L-\alpha_{t+1}^{-1})\beta_{t+1}^2\Vert z^{t+1}-y^{t+1}\Vert^2\nonumber\\
&&-2\alpha_{t+1}\beta_{t+1}\langle\epsilon^{t+1},\beta_{t+1}z^{t+1}-(\beta_{t+1}-1)z^t-x^*\rangle.\label{equation:error:bound:fista:eq4}
\end{eqnarray}

Concerning the last line of \eqref{equation:error:bound:fista:eq4}, we will rewrite it as
\begin{eqnarray}
\langle\epsilon^{t+1},\beta_{t+1}z^{t+1}-(\beta_{t+1}-1)z^t-x^*\rangle &=&
\langle\epsilon^{t+1},\beta_{t+1}z^{t+1}-\beta_{t+1}y^{t+1}\rangle\nonumber\\
&&+\langle\epsilon^{t+1},\beta_{t+1}y^{t+1}-(\beta_{t+1}-1)z^t-x^*\rangle.\label{equation:error:bound:fista:eq5}
\end{eqnarray}

Now, by the extrapolation \eqref{equation:fist:eq2} and definition \eqref{equation:def:sk}, we have
\begin{eqnarray}
\beta_{t+1}y^{t+1}-(\beta_{t+1}-1)z^t=\beta_tz^t-(\beta_t-1)z^{t-1}=s^t.
\label{equation:error:bound:fista:eq6}
\end{eqnarray}

Using definition \eqref{equation:def:sk} for index $t+1$ and \eqref{equation:error:bound:fista:eq5}-\eqref{equation:error:bound:fista:eq6} in \eqref{equation:error:bound:fista:eq4} we obtain 
\begin{eqnarray}
2\alpha_{t+1}\beta_{t+1}^2v_{t+1}-2\alpha_{t+1}\beta_t^2v_t &\le &
\Vert s^t-x^*\Vert^2-\Vert s^{t+1}-x^*\Vert^2\nonumber\\
&&+\alpha_{t+1}\beta_{t+1}^2(L-\alpha_{t+1}^{-1})\Vert y^{t+1}-z^{t+1}\Vert^2
+2\alpha_{t+1}\beta_{t+1}^2\langle\epsilon^{t+1},y^{t+1}-z^{t+1}\rangle\nonumber
\\
&&+2\alpha_{t+1}\beta_{t+1}\langle\epsilon^{t+1},x^*-s^t\rangle.\label{equation:error:bound:fista:eq7}
\end{eqnarray}

We now bound the second line in the above inequality as
\begin{eqnarray*}
-(\alpha_{t+1}^{-1}-L)\Vert y^{t+1}-z^{t+1}\Vert^2
+2\langle\epsilon^{t+1},y^{t+1}-z^{t+1}\rangle
&\le &\frac{\Vert\epsilon^{t+1}\Vert^2}{(\alpha_{t+1}^{-1}-L)},
\end{eqnarray*}
using Young's inequality $2\langle a,b\rangle\le\frac{\Vert a\Vert^2}{\lambda}+\lambda \Vert b\Vert^2$ with $a:=\epsilon^{t+1}$, $b:=y^{t+1}-z^{t+1}$ and $\lambda:=\alpha_{t+1}^{-1}-L>0$ by the assumption on $\{\alpha_t\}$. Using the above relation in \eqref{equation:error:bound:fista:eq7} we get
\begin{eqnarray*}
2\alpha_{t+1}\beta_{t+1}^2v_{t+1}-2\alpha_{t+1}\beta_t^2v_t &\le &
\Vert s^t-x^*\Vert^2-\Vert s^{t+1}-x^*\Vert^2\nonumber\\
&&+\frac{\alpha_{t+1}\beta_{t+1}^2}{(\alpha_{t+1}^{-1}-L)}\Vert\epsilon^{t+1}\Vert^2\nonumber\\
&&+2\alpha_{t+1}\beta_{t+1}\langle\epsilon^{t+1},x^*-s^t\rangle.
\end{eqnarray*}

Finally, we use the hypothesis that $\alpha_{t+1}\le\alpha_t$ above to prove the claimed recursion.\qed
\end{proof}

\begin{corollary}\label{cor:error:bound:fista}
Under the assumptions of Proposition \ref{proposition:error:bound:fista}, for any $1\le t\le T$ and $x^*\in \mbox{\emph{S}}(f,\varphi)$,
\begin{eqnarray*}
2\alpha_{T+1}\beta_{T+1}^2\left[g(z^{T+1})-g^*\right]+\Vert s^{T+1}-x^*\Vert^2 &\le &
2\alpha_t\beta_t^2\left[g(z^t)-g^*\right]+\Vert s^t-x^*\Vert^2
+\sum_{\tau=t}^T\Delta A_{\tau+1}+\sum_{\tau=t}^T\Delta M_{\tau+1}(x^*),
\end{eqnarray*}
where, for any $\tau\in\mathbb{N}$ and $x^*\in \mbox{\emph{S}}(f,\varphi)$, we have defined
\begin{eqnarray}
\Delta A_{\tau+1}&:=&\frac{\alpha_{\tau+1}^2\beta_{\tau+1}^2}{(1-L\alpha_{\tau+1})}\Vert\epsilon^{\tau+1}\Vert^2,\label{equation:A:t}\\
\Delta M_{\tau+1}(x^*)&:=&2\alpha_{\tau+1}\beta_{\tau+1}\langle\epsilon^{\tau+1},x^*-s^{\tau}\rangle.\label{equation:M:t}
\end{eqnarray}
Moreover, $\esp[\Delta M_{\tau+1}(x^*)]=0$ for all $\tau\in\mathbb{N}$.
\end{corollary}
\begin{proof}
Given, $1\le t\le T$ and $x^*\in\mbox{\emph{S}}(f,\varphi)$, we simply sum recursively the inequality in Proposition \ref{proposition:error:bound:fista} from $t$ to $T$ to obtain the first claim.

We prove the second claim by showing that $\{\Delta M_{t+1}(x^*),\alg_t\}$ defines a martingale difference, i.e., $\esp[\Delta M_{t+1}(x^*)|\alg_t]=0$ for all $t\in\mathbb{N}$. To show this, note that $y^{t+1}\in\alg_t$ since $z^{t-1},z^t\in\alg_t$. Moreover $\xi^{t+1}\perp\perp\alg_t$. Since $\epsilon^{t+1}=\nabla F(y^{t+1},\xi^{t+1})-\nabla f(y^{t+1})$, the previous statements imply that $\esp[\epsilon^{t+1}|\alg_t]=0$ and
$$
\esp[\Delta M_{t+1}(x^*)|\alg_t]=2\alpha_{t+1}\beta_{t+1}\langle\esp[\epsilon^{t+1}|\alg_t],x^*-s^t\rangle=0,
$$
where we also used that $s^t\in\alg_t$ since $z^{t-1},z^t\in\alg_t$. Using $\esp[\esp[\cdot|\alg_t]]=\esp[\cdot]$, we further conclude that $\esp[\Delta M_{t+1}(x^*)]=0$ as required.\qed
\end{proof}

\subsection{$\mathsf{L}^2$-boundedness of the iterates}

In the case $X$ is unbounded and the oracle has multiplicative noise, it is not possible to infer boundedness of $\{\Lnorm{\Vert z^t\Vert}\}_{t=1}^{\infty}$ a priori (i.e., $\mathsf{L}^2$-boundedness of the iterates). In this section we obtain such $\mathsf{L}^2$-boundedness when using stochastic approximation with dynamic mini-batches. This is essential to obtain complexity estimates in the following section.

\begin{proposition}\label{prop:l2:boundedness:smooth:convex}
Suppose the assumptions of Proposition \ref{proposition:error:bound:fista} hold and
$
\sum_{t=1}^{\infty}\frac{\alpha_{t+1}^2\beta_{t+1}^2}{(1-L\alpha_{t+1})N_{t+1}}<\infty.
$
Choose $t_0\in\mathbb{N}$ and $\gamma>0$ such that 
\begin{eqnarray}
\sum_{t\ge t_0}^{\infty}\frac{\alpha_{t+1}^2\beta_{t+1}^2}{(1-L\alpha_{t+1})N_{t+1}}<\gamma<\frac{1}{15\sigma_L^2}.\label{equation:sumability}
\end{eqnarray}

Then for all $x^*\in \mbox{\emph{S}}(f,\varphi)$,
\begin{eqnarray}
\sup_{t\ge 0}\Lnorm{\Vert z^t-x^*\Vert}^2
&\le &\frac{\max_{t\in[t_0]}\left\{2\alpha_{t_0}\beta_{t_0}^2\esp[g(z^{t_0})-g^*]+\esp[\Vert s^{t_0}-x^*\Vert^2]\right\}+\frac{\sigma(x^*)^2}{3\sigma_L^2}}{1-15\gamma\sigma_L^2}.\label{equation:l2:bound:z:t}
\end{eqnarray}
\end{proposition}
\begin{proof}
For clarity of exposition we use the notation $v^t:=\esp[g(z^t)-g^*]$. Given $t\ge t_0$, we take total expectation in the inequality of Corollary \ref{cor:error:bound:fista} and get, using $v^{t+1}\ge0$, 
\begin{eqnarray}
\esp[\Vert s^{t+1}-x^*\Vert^2]&\le & 2\alpha_{t+1}\beta_{k+1}^2v^{t+1}+\esp[\Vert s^{t+1}-x^*\Vert^2]\nonumber\\
&\le &2\alpha_{t_0}\beta_{t_0}^2v^{t_0}+\esp[\Vert s^{t_0}-x^*\Vert^2]+\sum_{i=t_0}^{t}\frac{\alpha_{i+1}\beta_{i+1}^2}{(\alpha_{i+1}^{-1}-L)}\esp[\Vert\epsilon^{i+1}\Vert^2].\label{equation:recursion:aux:fista}
\end{eqnarray}

Let $i\in\mathbb{N}$. From definition \eqref{equation:fist:eq2} we have
\begin{eqnarray}
y^{i+1}-x^*&=&\left(\frac{\beta_i-1}{\beta_{i+1}}+1\right)(z^i-x^*)-\left(\frac{\beta_i-1}{\beta_{i+1}}\right)(z^{i-1}-x^*).\label{equation:expression:yt}
\end{eqnarray}
We now use the above expression in Lemma \ref{lemma:oracle:variance:decay}. Precisely, we use Assumptions \ref{assump:iid:sampling}-\ref{assump:oracle:multiplicative:noise}, $\epsilon^{i+1}=F'(y^t,\xi^t)-\nabla f(y^t)$, definition of $F'(y^t,\xi^t)$ in \textsf{Algorithm \ref{algorithm:smooth:convex}}, $\xi^{i+1}\perp\perp\alg_{i+1}$ and $y^{i+1}\in\alg_{i+1}$. Then we get
\begin{eqnarray*}
\sqrt{N_{i+1}}\cdot\Lnorm{\Vert\epsilon^{i+1}\Vert|\alg_{i+1}}&\le &\sigma(x^*)+\sigma_L\Vert y^{i+1}-x^*\Vert\\
&\le &\sigma(x^*)+\sigma_L\left(\frac{\beta_i-1}{\beta_{i+1}}+1\right)\Vert z^i-x^*\Vert+\sigma_L\left(\frac{\beta_i-1}{\beta_{i+1}}\right)\Vert z^{i-1}-x^*\Vert.
\end{eqnarray*}
From the above, we use $\Lnorm{\Lnorm{\cdot|\alg_{i+1}}}=\Lnorm{\cdot}$ and take squares to get
\begin{eqnarray}
N_{i+1}\cdot\Lnorm{\Vert\epsilon^{i+1}\Vert}^2&\le &3\sigma(x^*)^2+3\sigma_L^2\left(\frac{\beta_i-1}{\beta_{i+1}}+1\right)^2\Lnorm{\Vert z^i-x^*\Vert}^2+3\sigma_L^2\left(\frac{\beta_i-1}{\beta_{i+1}}\right)^2\Lnorm{\Vert z^{i-1}-x^*\Vert}^2,\label{equation:bound:stochastic:error:fista}
\end{eqnarray}
where we used the relation $(\sum_{i=1}^3a_i)^2\le3\sum_{i=1}^3a_i^2$.

For simplicity, we define
$
q_i:=\frac{\beta_i-1}{\beta_{i+1}}
$
and 
$
d_i:=\Lnorm{\Vert z^i-x^*\Vert}.
$
 From \eqref{equation:recursion:aux:fista}, \eqref{equation:bound:stochastic:error:fista} and $0\le q_i\le1$ for all $i\in\mathbb{N}$, we finally get the recursion for any $t\ge t_0$,
\begin{eqnarray}
\esp[\Vert s^{t+1}-x^*\Vert^2]&\le &2\alpha_{t_0}\beta_{t_0}^2v^{t_0}+\esp[\Vert s^{t_0}-x^*\Vert^2]+\sum_{i=t_0}^{t}\frac{\alpha_{i+1}\beta_{i+1}^2}{(\alpha_{i+1}^{-1}-L)}\cdot\frac{3\sigma(x^*)^2}{N_{i+1}},\nonumber\\
&&+\sum_{i=t_0}^{t}\frac{\alpha_{i+1}\beta_{i+1}^2}{(\alpha_{i+1}^{-1}-L)}\cdot\frac{12\sigma_L^2}{N_{i+1}}d_i^2
+\sum_{i=t_0}^{t}\frac{\alpha_{i+1}\beta_{i+1}^2}{(\alpha_{i+1}^{-1}-L)}\cdot\frac{3\sigma_L^2}{N_{i+1}}d_{i-1}^2.\label{equation:recursion:aux2:fista}
\end{eqnarray}

For any $a>0$, we define the stopping time
\begin{eqnarray}
\tau_a:=\inf\left\{t\ge t_0:d_t>a\right\},\label{equation:def:tau:fista}
\end{eqnarray}
where $t_0$ and $\gamma$ are as defined in the statement of the proposition. Note that for any $t\in\mathbb{N}$,
\begin{eqnarray*}
s^{t+1}-x^*&=&\beta_{t+1}(z^{t+1}-x^*)-(\beta_{t+1}-1)(z^t-x^*).
\end{eqnarray*}
From the above equality, $\beta_{t+1}\ge1$, the triangle inequality for $\Vert\cdot\Vert$ and Minkowski's inequality for $\Lnorm{\cdot}$, we get
\begin{eqnarray*}
\Lnorm{\Vert s^{t+1}-x^*\Vert}\ge\beta_{t+1}\Lnorm{\Vert z^{t+1}-x^*\Vert}-(\beta_{t+1}-1)\Lnorm{\Vert z^t-x^*\Vert}.
\end{eqnarray*}
The above relation implies that for any $a>0$ such that $\tau_a<\infty$,
\begin{eqnarray}
\Lnorm{\Vert s^{\tau_a}-x^*\Vert}\ge\beta_{\tau_a}d_{\tau_a}-(\beta_{\tau_a}-1)d_{\tau_a-1}
>\beta_{\tau_a}a-(\beta_{\tau_a}-1)a=a,\label{equation:lower:bound:s:tau}
\end{eqnarray}
since $d_{\tau_a}>a$ and $d_{\tau_a-1}\le a$ by definition \eqref{equation:def:tau:fista}.\footnote{We note here the importance of assuming $\beta_t\ge1$ and the specific form of the extrapolation in \eqref{equation:fist:eq2} in terms of previous iterates.}
	
From \eqref{equation:sumability} and \eqref{equation:recursion:aux2:fista}-\eqref{equation:lower:bound:s:tau}, we have that for any $a>0$ such that $\tau_a<\infty$,
\begin{eqnarray*}
a^2<\esp[\Vert s^{\tau_a}-x^*\Vert^2]&\le &2\alpha_{t_0}\beta_{t_0}^2v^{t_0}+\esp[\Vert s^{t_0}-x^*\Vert^2]+\sum_{i=t_0}^{\tau_a-1}\frac{\alpha_{i+1}\beta_{i+1}^2}{(\alpha_{i+1}^{-1}-L)}\cdot\frac{3\sigma(x^*)^2}{N_{i+1}},\nonumber\\
&&+\sum_{i=t_0}^{\tau_a-1}\frac{\alpha_{i+1}\beta_{i+1}^2}{(\alpha_{i+1}^{-1}-L)}\cdot\frac{12\sigma_L^2}{N_{i+1}}d_i^2
+\sum_{i=t_0}^{\tau_a-1}\frac{\alpha_{i+1}\beta_{i+1}^2}{(\alpha_{i+1}^{-1}-L)}\cdot\frac{3\sigma_L^2}{N_{i+1}}d_{i-1}^2\\
&\le &2\alpha_{t_0}\beta_{t_0}^2v^{t_0}+\esp[\Vert s^{t_0}-x^*\Vert^2]+
3\gamma\left[\sigma(x^*)^2+5\sigma_L^2a^2\right],
\end{eqnarray*}
and hence,
\begin{eqnarray}\label{equation:sequence:treshold}
a^2<\frac{2\alpha_{t_0} \beta_{t_0}^2v^{t_0}+\esp[\Vert s^{t_0}-x^*\Vert^2]+3\gamma\sigma(x^*)^2}{1-15\gamma\sigma_L^2},
\end{eqnarray}
where we used that
$
0<\gamma<\frac{1}{15\sigma_L^2}.
$
By definition of $\tau_a$ for any $a>0$, the argument above shows that any threshold $a^2$ which the sequence $\{d_t^2\}_{t\ge t_0}$ eventually exceeds is bounded above by \eqref{equation:sequence:treshold}. Hence $\{d_t^2\}_{t\ge t_0}$ is bounded and
$$
\sup_{t\ge t_0}d_t^2\le \frac{2\alpha_{t_0}\beta_{t_0}^2v^{t_0}+\esp[\Vert s^{t_0}-x^*\Vert^2]+3\gamma\sigma(x^*)^2}{1-15\gamma\sigma_L^2}.
$$
Since the denominator above is less then $1$, the bound above implies further that
\begin{eqnarray*}
\sup_{t\ge 0}d_t^2&\le &\frac{\max_{t\in[t_0]}\left\{2\alpha_{t_0}\beta_{t_0}^2v^{t_0}+\esp[\Vert s^{t_0}-x^*\Vert^2]\right\}+3\gamma\sigma(x^*)^2}{1-15\gamma\sigma_L^2}\\
&\le &\frac{\max_{t\in[t_0]}\left\{2\alpha_{t_0}\beta_{t_0}^2v^{t_0}+\esp[\Vert s^{t_0}-x^*\Vert^2]\right\}+\frac{\sigma(x^*)^2}{5\sigma_L^2}}{1-15\gamma\sigma_L^2},
\end{eqnarray*}
where we used 
$
0<\gamma<\frac{1}{15\sigma_L^2}.
$
This concludes the proof of the claim.\qed
\end{proof}

\subsection{Convergence rate and oracle complexity}

We now derive a convergence rate and estimate the oracle complexity of \textsf{Algorithm \ref{algorithm:smooth:convex}}.

\begin{theorem}\label{thm:rate:smooth:convex:multiplicative:noise}
Suppose the Assumptions of Proposition \ref{proposition:error:bound:fista} hold and the sequence $\{y^t,z^t\}$ is generated by \textsf{\emph{Algorithm \ref{algorithm:smooth:convex}}}. Let $\mu\in(0,1)$, $a,b,\delta>0$, $N_0\in\mathbb{N}$ and set for all $t\in\mathbb{N}$,
\begin{eqnarray*}
\beta_t:=\frac{1+t}{2},\quad \alpha_t:=\frac{\mu}{L+\frac{a}{\sqrt{N_0}}},
\quad N_t:=N_0\Big\lfloor(t+2+\delta)^3\left[\ln(t+2+\delta)\right]^{1+2b}\Big\rfloor.
\end{eqnarray*}
Choose $\phi\in(0,1)$ and let $t_0:=t_0(\alpha_1\sigma_L,N_0,b,\delta)\in\mathbb{N}$ be given by
\begin{eqnarray}\label{equation:thm:rate:smooth:convex:multiplicative:noise:t0}
t_0:=\left\lceil\exp\left\{\sqrt[2b]{\frac{15(\alpha_1\sigma_L)^2}{8\phi N_0 b}}\right\}-1-\delta\right\rceil\bigvee1.
\end{eqnarray}

Then Proposition \ref{prop:l2:boundedness:smooth:convex} holds. Moreover, given $x^*\in\mbox{\emph{S}}(f,\varphi)$ and $\mathsf{J}:=\mathsf{J}(x^*,t_0)>0$ such that
\begin{equation*}
\sup_{\tau\in\mathbb{N}}\Lnorm{\Vert z^\tau-x^*\Vert}^2\le\mathsf{J},
\end{equation*}
the following bound holds for all $t\in\mathbb{N}$,
\begin{eqnarray*}
\esp\left[g(z^{t+1})-g^*\right]&\le &
\frac{4\esp\left[g(z^1)-g^*\right]}{(t+2)^2}+2\left[\frac{L}{\mu}+\frac{a}{\mu\sqrt{N_0}}\right]\frac{\esp\left[\Vert s^1-x^*\Vert^2\right]}{(t+2)^2}\\
&&+\left[\frac{L}{\mu}+\frac{a}{\mu\sqrt{N_0}}\right]\cdot\frac{3\mu^2}{4(1-\mu)a^2b\left[\ln(2+\delta)\right]^{2b}}\cdot\frac{\sigma(x^*)^2}{(t+2)^2}\\
&&+\left[\frac{L}{\mu}+\frac{a}{\mu\sqrt{N_0}}\right]\cdot\frac{15\mu^2}{4(1-\mu)N_0b\left[\ln(2+\delta)\right]^{2b}}\cdot\frac{(\sigma_L/L)^2\mathsf{J}}{(t+2)^2}.
\end{eqnarray*}
\end{theorem}
\begin{proof}
We first show that $\{\beta_t\}$, $\{\alpha_t\}$ and $\{N_t\}$ satisfy the conditions of Propositions \ref{proposition:error:bound:fista}-\ref{prop:l2:boundedness:smooth:convex}. We have that $\{\alpha_t\}$ is a constant (an hence non-increasing) sequence satisfying $0<L\alpha_t\le\mu$. By inspection, it is easy to check that $\beta_t\ge1$ and $\beta_t^2=\beta_{t+1}^2-\beta_t^2$ for all $t\ge1$. Also,
\begin{eqnarray*}
\sum_{t=1}^{\infty}\frac{\alpha_{t+1}\beta_{t+1}^2}{(\alpha_{t+1}^{-1}-L)N_{t+1}}
&\le &\frac{\alpha_1^2}{4(1-L\alpha_1)N_0}\sum_{t=1}^{\infty}\frac{1}{(t+2+\delta)[\ln(t+2+\delta)]^{1+2b}}<\infty.
\end{eqnarray*}
Hence, we have have shown that the policies $\{\beta_t\}$, $\{\alpha_t\}$ and $\{N_t\}$ satisfy the conditions of Propositions \ref{proposition:error:bound:fista}-\ref{prop:l2:boundedness:smooth:convex}.

For $\phi\in(0,1)$, we want $t_0$ to satisfy 
\begin{eqnarray}
\sum_{t\ge t_0}^{\infty}\frac{\alpha_{t+1}\beta_{t+1}^2}{(\alpha_{t+1}^{-1}-L)N_{t+1}}\le\frac{\phi}{15\sigma_L^2},\label{thm:rate:smooth:convex:multiplicative:noise:eq1}
\end{eqnarray}
as prescribed in Proposition \ref{prop:l2:boundedness:smooth:convex}. We have 
\begin{eqnarray}
\sum_{t\ge t_0}^{\infty}\frac{\alpha_{t+1}\beta_{t+1}^2}{(\alpha_{t+1}^{-1}-L)N_{t+1}}
&\le &\frac{\alpha_1^2}{4(1-L\alpha_1)N_0}\sum_{t\ge t_0}^{\infty}\frac{1}{(t+2+\delta)[\ln(t+2+\delta)]^{1+2b}}\nonumber\\
&\le &\frac{\alpha_1^2}{4(1-L\alpha_1)N_0}\int_{t_0-1}^{\infty}\frac{\dist t}{(t+2+\delta)[\ln(t+2+\delta)]^{1+2b}}\nonumber\\
&=&\frac{\alpha_1^2}{4(1-L\alpha_1)N_0}\cdot\frac{1}{2b\ln(t_0+1+\delta)^{2b}}\nonumber\\
&\le &\frac{\alpha_1^2}{8(1-\mu)N_0b\ln(t_0+1+\delta)^{2b}}.\label{thm:rate:smooth:convex:multiplicative:noise:eq2}
\end{eqnarray}
From the above relation, it is sufficient to choose $t_0$ as the minimum natural number such that the right hand side of \eqref{thm:rate:smooth:convex:multiplicative:noise:eq2} is less than $\phi/15\sigma_L^2$. This is satisfied by $t_0$ in \eqref{equation:thm:rate:smooth:convex:multiplicative:noise:t0}.

Let $x^*\in\mbox{S}(f,\varphi)$. From Proposition \ref{prop:l2:boundedness:smooth:convex} and \eqref{thm:rate:smooth:convex:multiplicative:noise:eq1} we know that $\{\Lnorm{\Vert z^\tau-x^*\Vert}\}_{\tau\ge1}$ is bounded, say
$
\sup_{\tau\in\mathbb{N}}\Lnorm{\Vert z^\tau-x^*\Vert}^2\le\mathsf{J}. 
$
From this, \eqref{equation:bound:stochastic:error:fista} and $\sup_{\tau\in\mathbb{N}}\frac{\beta_\tau-1}{\beta_{\tau+1}}\le1$ we get
\begin{eqnarray}
\esp\left[\Vert\epsilon^{\tau+1}\Vert^2\right]\le\frac{3\sigma(x^*)^2+15\sigma_L^2\mathsf{J}}{N_{\tau+1}}. 
\label{thm:rate:smooth:convex:multiplicative:noise:eq3}
\end{eqnarray}
We now bound the expectation of the sum $\sum_{t}\Delta A_t$ in the inequality of Corollary \ref{cor:error:bound:fista}. Precisely, for any $t\ge1$,
\begin{eqnarray*}
\sum_{\tau=1}^t\esp\left[\Delta A_\tau\right]&=&\sum_{\tau=1}^t\frac{\alpha_{\tau+1}^2\beta_{\tau+1}^2\esp[\Vert\epsilon^{\tau+1}\Vert^2]}{(1-L\alpha_{\tau+1})}\nonumber\\
&\le &\frac{\alpha_1^2}{4(1-\mu)}\sum_{\tau=1}^t(\tau+2)^2\esp\left[\Vert\epsilon^{\tau+1}\Vert^2\right]\nonumber\\
&\le &\frac{\alpha_1^2\left[3\sigma(x^*)^2+15\sigma_L^2\mathsf{J}\right]}{4(1-\mu)N_0}\sum_{\tau=1}^t\frac{(\tau+2)^2}{(\tau+2+\delta)^3\left[\ln(\tau+2+\delta)\right]^{1+2b}}\nonumber\\
&\le &\frac{\alpha_1^2\left[3\sigma(x^*)^2+15\sigma_L^2\mathsf{J}\right]}{4(1-\mu)N_0}\sum_{\tau=1}^t\frac{1}{(\tau+2+\delta)\left[\ln(\tau+2+\delta)\right]^{1+2b}}\nonumber\\
&\le &\frac{\alpha_1^2\left[3\sigma(x^*)^2+15\sigma_L^2\mathsf{J}\right]}{4(1-\mu)N_0}\int_{\tau=0}^t\frac{\dist \tau}{(\tau+2+\delta)\left[\ln(\tau+2+\delta)\right]^{1+2b}}\nonumber\\
&\le &\frac{\alpha_1^2\left[3\sigma(x^*)^2+15\sigma_L^2\mathsf{J}\right]}{8(1-\mu)N_0b\left[\ln(2+\delta)\right]^{2b}}=
\frac{3[\alpha_1\sigma(x^*)]^2+15(\alpha_1\sigma_L)^2\mathsf{J}}{8(1-\mu)N_0b\left[\ln(2+\delta)\right]^{2b}},
\end{eqnarray*}
where we used \eqref{thm:rate:smooth:convex:multiplicative:noise:eq3} and definition of $N_\tau$ in third inequality. 

From the above inequality and the bounds
\begin{eqnarray*}
[\alpha_1\sigma(x^*)]^2\le\left[\frac{\mu\sqrt{N_0}}{a}\sigma(x^*)\right]^2=\frac{\mu^2N_0\sigma(x^*)^2}{a^2},\quad\quad\quad
(\alpha_1\sigma_L)^2\le\left(\frac{\mu}{L}\sigma_L\right)^2=\frac{\mu^2\sigma_L^2}{L^2},
\end{eqnarray*}
we finally get
\begin{eqnarray}
\sum_{\tau=1}^t\esp\left[\Delta A_\tau\right]&\le &
\frac{3\mu^2\sigma(x^*)^2}{8(1-\mu)a^2b\left[\ln(2+\delta)\right]^{2b}}
+\frac{15\mu^2(\sigma_L/L)^2\mathsf{J}}{8(1-\mu)N_0b\left[\ln(2+\delta)\right]^{2b}}.
\label{thm:rate:smooth:convex:multiplicative:noise:eq4}
\end{eqnarray}

We also have for any $t\ge1$,
\begin{eqnarray}
\alpha_t^{-1}=\frac{L}{\mu}+\frac{a}{\mu\sqrt{N_0}},\quad\frac{1}{\beta_t^2}=\frac{4}{(t+1)^2}.
\label{thm:rate:smooth:convex:multiplicative:noise:eq5}
\end{eqnarray}

Finally, we take total expectation in the inequality of Corollary \ref{cor:error:bound:fista} for $t:=1$ and $T:=t$ and use $\esp[\Delta M_t(x^*)]=0$ for all $t\in\mathbb{N}$, $\beta_1=1$ and \eqref{thm:rate:smooth:convex:multiplicative:noise:eq4}-\eqref{thm:rate:smooth:convex:multiplicative:noise:eq5} to derive the required claim.\qed
\end{proof}

\begin{remark}
Regarding the constant $\mathsf{J}$ in Theorem \ref{thm:rate:smooth:convex:multiplicative:noise}, from Proposition \ref{prop:l2:boundedness:smooth:convex} we have the upper bound
\begin{eqnarray}
\mathsf{J}&\le &\frac{\max_{t\in[t_0]}\left\{\frac{1}{2}\alpha_1(t_0+1)^2\esp[g(z^{t_0})-g^*]+\esp[\Vert s^{t_0}-x^*\Vert^2]\right\}+\frac{\sigma(x^*)^2}{5\sigma_L^2}}{1-\phi}.\label{equation:thm:rate:smooth:convex:multiplicative:noise:J}
\end{eqnarray}
\end{remark}

\begin{corollary}\label{cor:rate:complexity:smooth:convex:multiplicative:noise}
Let the assumptions of Theorem \ref{thm:rate:smooth:convex:multiplicative:noise} hold. Given $\epsilon>0$, \textsf{\emph{Algorithm \ref{algorithm:smooth:convex}}} achieves the tolerance
$
\esp\left[g(z^{T})-g^*\right]\le\epsilon
$
after $T=\mathcal{O}(\epsilon^{-\frac{1}{2}})$ iterations using an oracle complexity of
$$
\sum_{\tau=1}^TN_{\tau}\le\mathcal{O}\left(\epsilon^{-2}\right)\left[\ln\left(\epsilon^{-\frac{1}{2}}\right)\right]^2.
$$
\end{corollary}
\begin{proof}
In Theorem \ref{thm:rate:smooth:convex:multiplicative:noise}, we set $b=\frac{1}{2}$. For every $t\in\mathbb{N}$, let $B_{t+1}$ be the right hand side expression in the bound of the optimality gap $\esp[g(z^{t+1})-g^*]$ stated in Theorem \ref{thm:rate:smooth:convex:multiplicative:noise}. Up to a constant $B>0$, for every $t\in\mathbb{N}$, we have
$$
\esp[g(z^{t})-g^*]\le B_t\le \frac{B}{t^2}.
$$
Given $\epsilon>0$, let $T$ be the least natural number such that $BT^{-2}\le\epsilon$. Then $T=\mathcal{O}(\epsilon^{-\frac{1}{2}})$, $\esp\left[g(z^T)-g^*\right]\le\epsilon$ and
\begin{eqnarray*}
\sum_{\tau=1}^TN_\tau\lesssim\sum_{\tau=1}^T\tau^3(\ln\tau)^2\lesssim T^4(\ln T)^2\lesssim\epsilon^{-2}\left[\ln\left(\epsilon^{-\frac{1}{2}}\right)\right]^2.
\end{eqnarray*}
We have thus proved the required claims.\qed
\end{proof}

\begin{remark}[Constants for the smooth non-strongly convex case] \label{remark:constants:smooth:convex}
We discuss the constants in the bounds of Theorem \ref{thm:rate:smooth:convex:multiplicative:noise} and Corollary \ref{cor:rate:complexity:smooth:convex:multiplicative:noise} and compare it to previous bounds under \eqref{equation:uniform:variance:intro}. The optimality gap rate in Theorem \ref{thm:rate:smooth:convex:multiplicative:noise} depends on \emph{$t_0$ initial iterates} (with possibly $t_0>1$) given in \eqref{equation:thm:rate:smooth:convex:multiplicative:noise:t0} and \eqref{equation:thm:rate:smooth:convex:multiplicative:noise:J}. This requirement is needed in Proposition \ref{prop:l2:boundedness:smooth:convex} since \emph{no boundedness of the oracle's variance is assumed} a priori. Another distinctive feature is the presence of the factor $(t_0+1)^2$ in \eqref{equation:thm:rate:smooth:convex:multiplicative:noise:J} as a consequence of \emph{acceleration} under Assumption \ref{assump:oracle:multiplicative:noise}. These observations require showing that $t_0$ in \eqref{equation:thm:rate:smooth:convex:multiplicative:noise:t0} is not too large. In that respect, we note that $t_0$ \emph{does not depend on} $x^*\in \mbox{S}(f,\varphi)$, but only on $(\alpha_1\sigma_L)^2$ and the exogenous parameters $\phi$, $N_0$, $b$ and $\delta$. If we assume the standard Lipschitz continuity \eqref{equation:lipschitz:random:intro} in Lemma \ref{lemma:multiplicative:noise:from:random lipschitz}, then $\sigma_L=2L$ for
$
L:=\Lnorm{\mathsf{L}(\xi)}.
$
Also, the stepsize satisfies $\alpha_1\le\frac{\mu}{L}$ so that $(\alpha_1\sigma_L)^2\le4\mu^2$.
 
We thus conclude: the iteration $t_0$ is dictated solely by the multiplicative per unit distance variance $\sigma_L^2$, independently of the variances $\{\sigma(x)^2\}_{x\in X}$ at the points of the feasible set $X$. Moreover, assuming $L$ is known for the stepsize policy, there exists an upper bound on $t_0$ which is also independent of $\sigma_L^2$ and only depends on the exogenous parameters $\mu$, $N_0$, $b$ and $\delta$ chosen on the stepsize and sampling rate policies. 

For instance, if we set $\phi:=\mu:=b:=\frac{1}{2}$, $N_0:=2$, then $t_0\sim\lceil 44.52-\delta\rceil$. 
We now set $\delta:=44$ so that $t_0=1$ and $(t_0+1)^2=4$. We further choose $a\sim L$. Then using \eqref{equation:thm:rate:smooth:convex:multiplicative:noise:J}, $\alpha_1\le\frac{\mu}{L}$, the bound in Theorem \ref{thm:rate:smooth:convex:multiplicative:noise} becomes of the form
\begin{eqnarray*}
\esp\left[g(z^{t+1})-g^*\right]&\lesssim &\frac{1}{t^2}\left\{\esp\left[g(z^1)-g^*+L\Vert s^1-x^*\Vert^2\right]+\frac{\sigma(x^*)^2}{L^2}\right\}.
\end{eqnarray*}
We note that we may further control $\esp[g(z^1)-g^*]$ in terms of $L\esp[\Vert z^1-x^*\Vert^2]$ by using inequality \eqref{equation:prox:and:approx:inequalities:SA} of Proposition \ref{proposition:error:bound:fista}. Using this observation, the above inequality resembles bounds obtained if it was supposed that \eqref{equation:uniform:variance:intro} holds \emph{but replacing $\sigma$ by $\sigma(x^*)$} (see \cite{ghadimi:lan2016}, Corollary 5, equations (3.38) and (3.40)). In this sense, we improve previous results by showing that under the more aggressive setting of Assumption \ref{assump:oracle:multiplicative:noise}, our bounds depends on \emph{local variances $\sigma(x^*)^2$} at points $x^*\in\mbox{S}(f,\varphi)$. Moreover, in case \eqref{equation:uniform:variance:intro} indeed holds, our bounds are sharper than in \cite{ghadimi:lan2016} since typically $\sigma(x^*)^2\ll\sigma^2$ for large $\sqrt{\diam(X)}$ (see Example \ref{example:1}). This may be seen as a \emph{localization property} of \textsf{Algorithm \ref{algorithm:smooth:convex}} in terms of the oracle's variance.
\end{remark}

\section{Smooth strongly convex optimization with multiplicative noise}\label{section:smooth:strongly:convex}
We propose \textsf{Algorithm \ref{algorithm:smooth:strongly:convex}} for the problem \eqref{problem:min:composite} assuming an stochastic oracle satisfying \eqref{equation:expected:valued:objective} and Assumption \ref{assump:oracle:multiplicative:noise} (multiplicative noise) in the case the objective $f$ is smooth strongly convex satisfying \eqref{equation:smoothness:intro} and \eqref{equation:strong:convexity}.
\begin{algorithm}
\caption{Stochastic proximal gradient method with dynamic mini-batching}\label{algorithm:smooth:strongly:convex}
\begin{algorithmic}[1]
  \scriptsize
  \STATE INITIALIZATION: initial iterate $x^1$, positive stepsize sequence $\{\alpha_t\}$ and sampling rate $\{N_t\}$.
  \STATE ITERATIVE STEP: Given $x^t$, generate i.i.d. sample $\xi^t:=\{\xi_j^t\}_{j=1}^{N_t}$ from $\probn$  independently from previous samples. Compute
  $$ 
  F_t'(x^t,\xi^t):=\frac{1}{N_t}\sum_{j=1}^{N_t}\nabla F(x^t,\xi^t_j).
  $$
  Then set
  	\begin{eqnarray}
	x^{t+1}:=P^{\alpha_t\varphi}_{x^t}\left[\alpha_tF'(x^t,\xi^t)\right]
	=\argmin_{x\in X}\left\{\ell_f\left(x^t,F'(x^t,\xi^t);x\right)+\frac{1}{2\alpha_t}\Vert x-x^t\Vert^2+\varphi(x)\right\}.\label{equation:algo:strongly:convex}
	\end{eqnarray}
\end{algorithmic}
\end{algorithm}
For $t\in\mathbb{N}$, we define the stochastic error
$
\epsilon^t := F_t'(x^t,\xi^t)-\nabla f(x^t)
$
and the filtration
$
\alg_t:=\sigma(x^1,\xi^1\ldots,\xi^{t-1}).
$

\subsection{Derivation of error bounds}

\begin{proposition}\label{prop:bound:strong:convex}
Suppose Assumption \ref{assump:iid:sampling}-\ref{assump:oracle:multiplicative:noise} hold for the problem \eqref{problem:min:composite} satisfying \eqref{equation:expected:valued:objective} as well as  \eqref{equation:smoothness:intro}-\eqref{equation:strong:convexity}. Let $x^*$ be the unique solution of $\eqref{problem:min:composite}$. Suppose $0<\alpha_t<\frac{1}{L}$ for all $t\in\mathbb{N}$.

Then the sequence generated by \textsf{\emph{Algorithm \ref{algorithm:smooth:strongly:convex}}} satisfies for all $t\in\mathbb{N}$, 
\begin{eqnarray*}
\esp\left[\Vert x^{t+1}-x^*\Vert^2|\alg_t\right]&\le &\left[1-c\alpha_t+\frac{2(\alpha_t\sigma_L)^2}{(1-L\alpha_t)N_t}\right]\Vert x^t-x^*\Vert^2+\frac{2\alpha_t^2\sigma(x^*)^2}{(1-L\alpha_t)N_t},\\
\esp\left[g(x^{t+1})-g^*\right]&\le &\left[\frac{\left(\alpha_t^{-1}-c\right)}{2}+\frac{2\alpha_t\sigma_L^2}{(1-L\alpha_t)N_t}\right]\esp\left[\Vert x^t-x^*\Vert^2\right]+\frac{2\alpha_t\sigma(x^*)^2}{(1-L\alpha_t)N_t}.
\end{eqnarray*}
\end{proposition}
\begin{proof}
We use \eqref{equation:algo:strongly:convex} and Lemma \ref{lemma:proximal inequality} with the convex function
$
p:=\ell_f\left(x^t,F'(x^t,\xi^t);\cdot\right)+\varphi
$
and $\alpha:=\alpha_t$, $y:=x^t$, $z:=x^{t+1}$ and $x:=x^*$ obtaining
\begin{eqnarray}
\frac{1}{2\alpha_t}\Vert x^*-x^{t+1}\Vert^2 &\le &\frac{1}{2\alpha_t}\Vert x^*-x^{t}\Vert^2+\ell_{f}(x^t,F'_t(x^t,\xi^t);x^*)+\varphi(x^*)-\ell_f(x^t, F'_t(x^t,\xi^t);x^{t+1})-\varphi(x^{t+1})\nonumber\\
&&-\frac{1}{2\alpha_t}\Vert x^{t+1}-x^t\Vert^2\nonumber\\
&=&\frac{1}{2\alpha_t}\Vert x^*-x^{t}\Vert^2+\ell_{f}(x^t,\nabla f(x^t);x^*)+\varphi(x^*)-\ell_f(x^t,\nabla f(x^t);x^{t+1})-\varphi(x^{t+1})\nonumber\\
&&-\frac{1}{2\alpha_t}\Vert x^{t+1}-x^t\Vert^2+\langle\epsilon^t,x^*-x^{t}\rangle+\langle\epsilon^t,x^t-x^{t+1}\rangle\nonumber\\
&\le &\frac{1}{2\alpha_t}\Vert x^*-x^{t}\Vert^2+\left[f(x^*)+\varphi(x^*)\right]-\frac{c}{2}\Vert x^*-x^t\Vert^2-\left[f(x^{t+1})+\varphi(x^{t+1})\right]+\frac{L}{2}\Vert x^{t+1}-x^t\Vert^2\nonumber\nonumber\\
&&-\frac{1}{2\alpha_t}\Vert x^{t+1}-x^t\Vert^2+\langle\epsilon^t,x^*-x^{t}\rangle+\langle\epsilon^t,x^t-x^{t+1}\rangle\nonumber\\
&=&\frac{(1-c\alpha_t)}{2\alpha_t}\Vert x^*-x^{t}\Vert^2+g^*-g(x^{t+1})-\frac{(1-L\alpha_t)}{2\alpha_t}\Vert x^{t+1}-x^t\Vert^2+\langle\epsilon^t,x^t-x^{t+1}\rangle\nonumber\\
&&+\langle\epsilon^t,x^*-x^{t}\rangle\nonumber\\
&\le &\frac{(1-c\alpha_t)}{2\alpha_t}\Vert x^*-x^{t}\Vert^2+\frac{\alpha_t}{2(1-L\alpha_t)}\Vert\epsilon^t\Vert^2+\langle\epsilon^t,x^*-x^{t}\rangle,\label{equation:prop:error:bound:strongly:convex:eq0}
\end{eqnarray}
where we used \eqref{equation:linearization:directed}-\eqref{equation:linearization} and definition of $\epsilon^t$ in the first equality, the lower and upper inequalities of Lemma \ref{lemma:inner:upper:approximation:inequality} for $p:=f$ (by strong convexity and smoothness of $f$) in second inequality while in the last inequality we used $g^*-g(x^{t+1})\le0$ and Young's inequality $2\langle\epsilon^t,x^t-x^{t+1}\rangle\le\lambda^{-1}\Vert\epsilon^t\Vert^2+\lambda\Vert x^{t+1}-x^t\Vert^2 $ with $\lambda:=\frac{1-L\alpha_t}{\alpha_t}>0$ (since $0<L\alpha_t<1$ by assumption).

We now observe that $x^t\in\alg_t$, $\xi^t\perp\perp\alg_t$, $\epsilon^t=F'(x^t,\xi^t)-\nabla f(x^t)$, Assumption \ref{assump:iid:sampling} and \eqref{equation:expected:valued:objective} imply that $\esp[\langle\epsilon^t, x^*-x^t\rangle|\alg_t]=0$. Using this observation and $x^t\in\alg_t$, we have that
\begin{eqnarray}
\esp\left[\frac{\alpha_t}{2(1-L\alpha_t)}\Vert\epsilon^t\Vert^2+\langle\epsilon^t,x^*-x^{t}\rangle\Bigg|\alg_t\right]&\le &\frac{2\alpha_t}{(1-L\alpha_t)N_t}\left[\sigma(x^*)^2+\sigma_L^2\Vert x^t-x^*\Vert^2\right],\label{equation:prop:error:bound:strongly:convex:eq0:2}
\end{eqnarray}
where we have used the relation $(a+b)^2\le2a^2+2b^2$ and
$
\Lnorm{\Vert\epsilon^t\Vert|\alg_t}\le\frac{\sigma(x^*)+\sigma_L\Vert x^t-x^*\Vert}{\sqrt{N_t}},
$
which follows from Assumption \ref{assump:iid:sampling}-\ref{assump:oracle:multiplicative:noise}, $\epsilon^t=\sum_{j=1}^{N_t}\frac{\nabla F(x^t,\xi_j^t)-\nabla f(x^t)}{N_t}$, relation \eqref{equation:expected:valued:objective} and Lemma \ref{lemma:oracle:variance:decay}.

We take $\esp[\cdot|\alg_t]$, use $x^t\in\alg_t$ and multiply by $2\alpha_t$ in \eqref{equation:prop:error:bound:strongly:convex:eq0} and use \eqref{equation:prop:error:bound:strongly:convex:eq0:2} to get
\begin{eqnarray}
\esp\left[\Vert x^{t+1}-x^*\Vert^2|\alg_t\right]&\le &(1-c\alpha_t)\Vert x^*-x^{t}\Vert^2+\frac{2\alpha_t^2}{(1-L\alpha_t)N_t}\left[\sigma(x^*)^2+\sigma_L^2\Vert x^t-x^*\Vert^2\right]\nonumber\\
&=&\left[1-c\alpha_t+\frac{2\alpha_t^2\sigma_L^2}{(1-L\alpha_t)N_t}\right]\Vert x^t-x^*\Vert^2
+\frac{2\alpha_t^2\sigma(x^*)^2}{(1-L\alpha_t)N_t}.\label{equation:prop:error:bound:strongly:convex:eq1}
\end{eqnarray}
To conclude, we take total expectation above and use the hereditary property $\esp[\esp[\cdot|\alg_t]]=\esp[\cdot]$.

We now prove the second claim. We use the upper inequality of Lemma \ref{lemma:inner:upper:approximation:inequality} with $p:=f$ (by smoothness of $f$) and obtain
\begin{eqnarray}
g(x^{t+1})&\le &\ell_f(x^t;x^{t+1})+\varphi(x^{t+1})+\frac{L}{2}\Vert x^{t+1}-x^t\Vert^2\nonumber\\
&=&\ell_f(x^t,F'(x^t,\xi^t);x^{t+1})+\varphi(x^{t+1})+\frac{L}{2}\Vert x^{t+1}-x^t\Vert^2+\langle-\epsilon^t,x^{t+1}-x^t\rangle\nonumber\\
&\le &\ell_f(x^t,F'(x^t,\xi^t);x^*)+\varphi(x^*)+\frac{\alpha_k^{-1}}{2}\Vert x^*-x^t\Vert^2-\frac{(\alpha_t^{-1}-L)}{2}\Vert x^{t+1}-x^t\Vert^2-\frac{\alpha_t^{-1}}{2}\Vert x^*-x^{t+1}\Vert^2\nonumber\\
&&+\langle\epsilon^t,x^{t}-x^{t+1}\rangle\nonumber\\
&=&\ell_f(x^t;x^*)+\varphi(x^*)+\frac{\alpha_t^{-1}}{2}\Vert x^*-x^t\Vert^2-\frac{(\alpha_t^{-1}-L)}{2}\Vert x^{t+1}-x^t\Vert^2-\frac{\alpha_t^{-1}}{2}\Vert x^*-x^{t+1}\Vert^2\nonumber\\
&&+\langle\epsilon^t,x^*-x^t\rangle+\langle\epsilon^t,x^t-x^{t+1}\rangle\nonumber\\
&\le &g(x^*)-\frac{c}{2}\Vert x^*-x^t\Vert^2+\frac{\alpha_t^{-1}}{2}\Vert x^*-x^t\Vert^2+\langle\epsilon^t,x^*-x^t\rangle
-\frac{(1-L\alpha_t)}{2\alpha_t}\Vert x^{t+1}-x^t\Vert^2+\langle\epsilon^t,x^t-x^{t+1}\rangle\nonumber\\
&\le &g(x^*)+\frac{\left(\alpha_t^{-1}-c\right)}{2}\Vert x^*-x^t\Vert^2+\langle\epsilon^t,x^*-x^t\rangle
+\frac{\alpha_t\Vert\epsilon^t\Vert^2}{2(1-L\alpha_t)},\label{equation:prop:error:bound:strongly:convex:eq2}
\end{eqnarray}
where we used \eqref{equation:linearization:directed}-\eqref{equation:linearization}, definition of the prox-mapping and definition of $\epsilon^t$ in the equalities, the expression \eqref{equation:algo:strongly:convex} and Lemma \ref{lemma:proximal inequality} with $\alpha:=\alpha_t$, $y:=x^t$, $z:=x^{t+1}$, $x:=x^*$ and the convex function
$
p:=\ell_f\left(x^t,F'(x^t,\xi^t);\cdot\right)+\varphi
$
in the second inequality, the lower inequality of Lemma \ref{lemma:inner:upper:approximation:inequality} with $p:=f$ (by strong convexity of $f$) in third inequality while in last inequality we used Young's inequality $2\langle\epsilon^t,x^t-x^{t+1}\rangle\le\lambda^{-1}\Vert\epsilon^t\Vert^2+\lambda\Vert x^{t+1}-x^t\Vert^2 $ with $\lambda:=\frac{1-L\alpha_t}{\alpha_t}>0$ (since $0<L\alpha_t<1$ by assumption).

We take $\esp[\cdot|\alg_t]$ and use $x^t\in\alg_t$ in \eqref{equation:prop:error:bound:strongly:convex:eq2} and then further use \eqref{equation:prop:error:bound:strongly:convex:eq0:2} to finally obtain
\begin{eqnarray*}
\esp\left[g(x^{t+1})-g^*|\alg_t\right]&\le &\frac{\left(\alpha_t^{-1}-c\right)}{2}\Vert x^*-x^t\Vert^2
+\frac{2\alpha_t}{(1-L\alpha_t)N_t}\left[\sigma(x^*)^2+\sigma_L^2\Vert x^t-x^*\Vert^2\right]
\\
&=&\left[\frac{\left(\alpha_t^{-1}-c\right)}{2}+\frac{2\alpha_t\sigma_L^2}{(1-L\alpha_t)N_t}\right]\Vert x^t-x^*\Vert^2+\frac{2\alpha_t\sigma(x^*)^2}{(1-L\alpha_t)N_t}.
\end{eqnarray*}
Finally, we take $\esp[\cdot|\alg_t]$ again and use $\esp[\esp[\cdot|\alg_t]]=\esp[\cdot]$ to finish the proof.\qed
\end{proof}

\subsection{Convergence rate and oracle complexity}

We now derive a convergence rate and estimate the oracle complexity of \textsf{Algorithm \ref{algorithm:smooth:strongly:convex}}.

\begin{theorem}\label{thm:rate:strong:convex}
Suppose the assumptions of Proposition \ref{prop:bound:strong:convex} hold. Let $\mu\in(0,1)$ and $\phi\in(0,1)$ such that
$
\phi\in(0,1-\mu c/L).
$
Choose $\zeta\in(0,1)$ and $N_0\in\mathbb{N}$ and set the stepsize and sampling rate sequences as
$$
\alpha_t\equiv\alpha:=\frac{\mu}{L},\quad N_t=N_0\lfloor\zeta^{-t}\rfloor,\quad (t\in\mathbb{N}).
$$
Define 
$
\rho:=\left(1-\mu\frac{c}{2L}+\phi\right)\bigvee\zeta<1.
$
Let
\begin{eqnarray}
t_0&:=&\left\lceil\log_{\frac{1}{\zeta}}\left(\frac{2\mu^2}{(1-\mu)\phi N_0}\cdot\frac{\sigma_L^2}{L^2}\right)\right\rceil\bigvee1.\label{equation:t0:strongly:convex}
\end{eqnarray}

Then there exists constant $\mathsf{C}>0$ such that
\begin{eqnarray}
\esp[\Vert x^{t+1}-x^*\Vert^2]&\le &\mathsf{C}\rho^{t+1},\quad\forall t\ge1\label{equation:linear:rate:iterates:from:t=1}\\
\esp\left[g(x^{t+1})-g^*\right]&\le &\left[\frac{\left(L\mu^{-1}-c\right)}{2}+\frac{2\mu}{(1-\mu)N_0}\cdot\frac{\sigma_L^2}{L}\zeta^t\right]\mathsf{C}\rho^t+\frac{2\mu}{(1-\mu)N_0}\cdot\frac{\sigma(x^*)^2}{L}\zeta^t,\quad\forall t\ge2\label{equation:linear:rate:optimality:from:t=2}
\end{eqnarray}
\end{theorem}
\begin{proof}
For simplicity of notation, in the following we set 
$
v_t:=\esp\left[\Vert x^t-x^*\Vert^2\right], 
$
and define
\begin{eqnarray}
\beta:=\frac{2\alpha^2\sigma(x^*)^2}{(1-L\alpha)N_0},\quad\delta:=\frac{2(\alpha\sigma_L)^2}{(1-L\alpha)N_0},\quad\lambda:=1-c\alpha.\label{equation:thm:rate:strong:convex:beta:delta:lambda}
\end{eqnarray}
Note that from the definitions of $\lambda$ and $\alpha=\frac{\mu}{L}$, we get 
$
\lambda=1-\mu\frac{c}{L}\in(0,1).
$

From the first recursion in Proposition \ref{prop:bound:strong:convex}, $N_t^{-1}\le N_0^{-1}\zeta^t$ and the above definitions, we have for all $t\in\mathbb{N}$,
\begin{eqnarray}\label{equation:thm:rate:strong:convex1}
v_{t+1}\le\left(\lambda+\delta\zeta^t\right)v_t+\beta\zeta^t.
\end{eqnarray}
Moreover, by definitions of $t_0$, $\delta$, $\lambda$, $\rho$ and $\alpha=\mu/L$ we have for all $t\ge t_0$,
\begin{eqnarray}\label{equation:thm:rate:strong:convex2}
\lambda<\lambda+\delta\zeta^t\le\lambda+\delta\zeta^{t_0}\le\lambda+\phi<\rho.
\end{eqnarray}

We now claim that for all $t\ge t_0$,
\begin{eqnarray}\label{equation:thm:rate:strong:convex3}
v_{t+1}\le(\lambda+\phi)^{t+1-t_0}v_{t_0}+\beta\sum_{\tau=0}^{t-t_0}(\lambda+\phi)^{\tau}\zeta^{t-\tau}.
\end{eqnarray}
We prove the claim by induction. Indeed, for $t:=t_0$ the claim \eqref{equation:thm:rate:strong:convex3} follows from \eqref{equation:thm:rate:strong:convex1}. Supposing \eqref{equation:thm:rate:strong:convex3} holds for $t\ge t_0$, then again by \eqref{equation:thm:rate:strong:convex1},
\begin{eqnarray}
v_{t+2}&\le &\left(\lambda+\delta\zeta^{t+1}\right)v_{t+1}+\beta\zeta^{t+1}\nonumber\\
&\le &\left(\lambda+\phi\right)\left[(\lambda+\phi)^{t+1-t_0}v_{t_0}+\beta\sum_{\tau=0}^{t-t_0}(\lambda+\phi)^{\tau}\zeta^{t-\tau}\right]+\beta\zeta^{t+1}\nonumber\\
&=&\left(\lambda+\phi\right)^{t+2-t_0}v_{t_0}+\beta\sum_{\tau=1}^{t+1-t_0}(\lambda+\phi)^\tau\zeta^{t+1-\tau}+\beta\zeta^{t+1}\nonumber\\
&=&\left(\lambda+\phi\right)^{t+2-t_0}v_{t_0}+\beta\sum_{\tau=0}^{t+1-t_0}(\lambda+\phi)^\tau\zeta^{t+1-\tau},
\label{equation:thm:rate:strong:convex3'}
\end{eqnarray}
where we used \eqref{equation:thm:rate:strong:convex2} in second inequality. This shows \eqref{equation:thm:rate:strong:convex3} holds for $t+1$. The claim \eqref{equation:thm:rate:strong:convex3} is hence proved.

We will now bound the sum in \eqref{equation:thm:rate:strong:convex3}. First, we note that $\rho\ge1-\mu\frac{c}{2L}+\phi$ and $\lambda+\phi<\rho$ which imply
\begin{eqnarray}\label{equation:thm:rate:strong:convex4}
\rho\ge\lambda+\phi+\mu\frac{c}{2L}\Rightarrow\frac{1}{1-\frac{\lambda+\phi}{\rho}}\le\frac{2L\rho}{c\mu}.
\end{eqnarray}
For $t\ge t_0$, we have
\begin{eqnarray}
\beta\sum_{\tau=0}^{t-t_0}(\lambda+\phi)^{\tau}\zeta^{t-\tau}&\le &
\frac{2\alpha^2\sigma(x^*)^2}{(1-L\alpha)N_0}\sum_{\tau=0}^{t-t_0}(\lambda+\phi)^{\tau}\rho^{t-\tau}\nonumber\\
&=&\frac{2\alpha^2\sigma(x^*)^2}{(1-L\alpha)N_0}\rho^t\sum_{\tau=0}^{t-t_0}\left(\frac{\lambda+\phi}{\rho}\right)^{\tau}\nonumber\\
&\le &\frac{2\alpha^2\sigma(x^*)^2}{(1-L\alpha)N_0}\cdot\frac{\rho^t}{1-\frac{\lambda+\phi}{\rho}}\nonumber\\
&\le &\frac{4\mu\sigma(x^*)^2}{(1-\mu)N_0Lc}\rho^{t+1},\label{equation:thm:rate:strong:convex5}
\end{eqnarray}
where we used definition of $\beta$ and $\zeta\le\rho$ in first inequality, \eqref{equation:thm:rate:strong:convex2} in the second inequality and $\alpha=\mu/L$ and \eqref{equation:thm:rate:strong:convex4} in the last one.

From recursion \eqref{equation:thm:rate:strong:convex3} and the bound \eqref{equation:thm:rate:strong:convex5}, we get for $t\ge t_0$,
\begin{eqnarray}
v_{t+1}&\le &(\lambda+\phi)^{t+1-t_0}v_{t_0}+\frac{4\mu\sigma(x^*)^2}{(1-\mu)N_0Lc}\rho^{t+1}\label{equation:thm:rate:strong:convex6}\\
&\le &\rho^{t+1}(\lambda+\phi)^{-t_0}v_{t_0}+\frac{4\mu\sigma(x^*)^2}{(1-\mu)N_0Lc}\rho^{t+1}
=\mathsf{C}_0\rho^{t+1},\nonumber
\end{eqnarray}
using \eqref{equation:thm:rate:strong:convex2} in the second inequality and the definitions of $\lambda$ in \eqref{equation:thm:rate:strong:convex:beta:delta:lambda} and $\mathsf{C}_0$ in \eqref{equation:c0:strongly:convex} in the equality. 

The above relation implies that the sequence $\{v_t\}$ has linear convergence once the iteration $t_0+1$ is achieved. By changing the constants $\mathsf{C}_0$ and $\rho$ properly, this implies that the whole sequence $\{v_t\}$ has linear convergence. For our purposes (see Remark \ref{remark:constants:strongly:convex}), we next derive a refined bound in terms of constants showing that linear convergence is obtained from the initial iteration when \emph{only} $\mathsf{C}_0$ is changed in \eqref{equation:thm:rate:strong:convex6} in a prescribed way.

From \eqref{equation:thm:rate:strong:convex1}, we also have for $1\le t<t_0$,
\begin{eqnarray*}
v_{t+1}\le(\lambda+\delta\zeta^t)v_t+\beta\zeta^t
\le\lambda v_t+\left(\delta\max_{\tau\in[t_0-1]}v_\tau+\beta\right)\zeta^t.
\end{eqnarray*}
We may use the above relation and proceed by induction analogously to \eqref{equation:thm:rate:strong:convex3}-\eqref{equation:thm:rate:strong:convex3'} to get for $1\le t<t_0$,
\begin{eqnarray}\label{equation:thm:rate:strong:convex7}
v_{t+1}&\le &\lambda^tv_1+\left(\delta\max_{\tau\in[t_0-1]}v_\tau+\beta\right)\sum_{\tau=0}^{t-1}\lambda^\tau\zeta^{t-\tau}
\end{eqnarray}
The sum above is bounded by
\begin{eqnarray}\label{equation:thm:rate:strong:convex8}
\sum_{\tau=0}^{t-1}\lambda^\tau\zeta^{t-\tau}\le\sum_{\tau=0}^{t-1}\lambda^\tau\rho^{t-\tau}
=\rho^t\sum_{\tau=0}^{t-1}\left(\frac{\lambda}{\rho}\right)^\tau
\le\frac{\rho^t}{1-\frac{\lambda}{\rho}}<\frac{2L\rho^{t+1}}{c\mu},
\end{eqnarray}
where we used $\zeta\le\rho$ in first inequality, \eqref{equation:thm:rate:strong:convex2} in second inequality and \eqref{equation:thm:rate:strong:convex4} in the last one. From \eqref{equation:thm:rate:strong:convex7}-\eqref{equation:thm:rate:strong:convex8}, we get for $1\le t<t_0$,
\begin{eqnarray}
v_{t+1}&\le &\rho^{t+1}\lambda^{-1}v_1+\frac{2L}{c\mu}\left(\delta\max_{\tau\in[t_0-1]}v_\tau+\beta\right)\rho^{t+1}\nonumber\\
&=&\rho^{t+1}\left[\lambda^{-1}v_1+\frac{4\left(\frac{L}{c}\right)(\alpha\sigma_L)^2}{\mu(1-\mu)N_0}\max_{\tau\in[t_0-1]}v_\tau+\frac{4\left(\frac{L}{c}\right)\alpha^2\sigma(x^*)^2}{\mu(1-\mu)N_0}\right]=\mathsf{C}_1\rho^{t+1}\label{equation:thm:rate:strong:convex9},
\end{eqnarray}
where we used \eqref{equation:thm:rate:strong:convex2} in the inequality, definitions of $\delta$, $\beta$, $\alpha=\mu/L$ and 
\begin{eqnarray}
\mathsf{C}_1:=\lambda^{-1}v_1+\frac{4\mu\sigma_L^2}{(1-\mu)N_0Lc}\cdot\max_{\tau\in[t_0-1]}v_\tau+\frac{4\mu\sigma(x^*)^2}{(1-\mu)N_0Lc},\label{equation:c1:strongly:convex}
\end{eqnarray}
in the equalities.

From \eqref{equation:thm:rate:strong:convex9}, we have in particular $v_{t_0}\le\mathsf{C}_1\rho^{t_0}$. Using this and \eqref{equation:thm:rate:strong:convex6}, we get for $t\ge t_0$,
\begin{eqnarray}
v_{t+1}&\le &(\lambda+\phi)^{t+1-t_0}\mathsf{C}_1\rho^{t_0}+\frac{4\mu\sigma(x^*)^2}{(1-\mu)N_0Lc}\rho^{t+1}\nonumber\\
&\le &\rho^{t+1-t_0}\mathsf{C}_1\rho^{t_0}+\frac{4\mu\sigma(x^*)^2}{(1-\mu)N_0Lc}\rho^{t+1}=\mathsf{C}\rho^{t+1},
\label{equation:thm:rate:strong:convex10}
\end{eqnarray}
using \eqref{equation:thm:rate:strong:convex2} in the second inequality and definitions of $\mathsf{C}_1$ in \eqref{equation:c1:strongly:convex} and of $\mathsf{C}$ in \eqref{equation:c:strongly:convex} in the equality. 

From relation \eqref{equation:thm:rate:strong:convex10} established for $t\ge t_0$, relation \eqref{equation:thm:rate:strong:convex9} established for $1\le t<t_0$ and $\mathsf{C}_1=\mathsf{C}-\frac{4\mu\sigma(x^*)^2}{(1-\mu)N_0Lc}<\mathsf{C}$, we finally prove \eqref{equation:linear:rate:iterates:from:t=1}.

To prove the second claim in \eqref{equation:linear:rate:optimality:from:t=2}, we use $L\alpha_t=\mu$, $N_t=N_0\lfloor\zeta^{-t}\rfloor$ and the derived relation \eqref{equation:linear:rate:iterates:from:t=1} for $\esp[\Vert x^t-x^*\Vert^2]$ in the second recursion of Proposition \ref{prop:bound:strong:convex} which bounds $\esp[g(x^{t+1})-g^*]$ in terms of $\esp[\Vert x^t-x^*\Vert^2]$.\qed
\end{proof}

\begin{corollary}\label{cor:rate:complexity:strong:convex}
Let assumptions of Theorem \ref{thm:rate:strong:convex} hold. Then there exists constant $\mathsf{A}>0$ such that for given $\epsilon>0$, \textsf{\emph{Algorithm \ref{algorithm:smooth:strongly:convex}}} achieves the tolerance
$
\esp\left[g(x^{T})-g^*\right]\le\epsilon
$
after $T=\mathcal{O}\left(\log_{\frac{1}{\rho}}(\mathsf{A}\epsilon^{-1})\right)$ iterations using an oracle complexity of
$
\sum_{\tau=1}^TN_{\tau}\le\frac{N_0\rho\cdot\mathcal{O}\left(\epsilon^{-1}\right)}{1-\zeta}\cdot\left(\frac{\rho}{\zeta}\right)^{\log_{\frac{1}{\rho}}(\mathsf{A}\epsilon^{-1})+1}.
$

In particular, if the stepsize parameter $\mu$ and sampling rate parameter $\zeta$ satisfy
$
\zeta=1-\mathsf{a}\left(\mu\frac{c}{2L}+\phi\right)
$
for some $\mathsf{a}\in(0,1]$, then the oracle complexity is 
$$
\sum_{\tau=1}^TN_{\tau}\le\frac{N_0\rho\cdot\mathcal{O}\left(\epsilon^{-1}\right)}{\mathsf{a}\left(\mu\frac{c}{2L}+\phi\right)}.
$$
Up to constants, the same oracle complexity is achieved for the iteration error $\esp[\Vert x^T-x^*\Vert^2]$.
\end{corollary}
\begin{proof}
We only prove the result to $\esp[g(x^T)-g^*]$ as the proof for $\esp[\Vert x^T-x^*\Vert^2]$ is similar. Set $\epsilon>0$ and let $T$ be minimum number of iterations for $\esp[g(x^T)-g^*]\le\epsilon$ to hold. We have
\begin{eqnarray}\label{equation:cor:rate:complexity:strong:convex1}
\sum_{\tau=1}^TN_\tau &\lesssim &N_0\sum_{\tau=1}^T(\zeta^{-1})^\tau=N_0\frac{(\zeta^{-1})^T-1}{1-\zeta}
\le N_0\frac{(\zeta^{-1})^T}{1-\zeta}.
\end{eqnarray}

From Theorem \ref{thm:rate:strong:convex}, for some constant $\mathsf{A}>0$,
$
T\le\log_{\frac{1}{\rho}}\left(\mathsf{A}\epsilon^{-1}\right)+1.
$
Using this we get
\begin{eqnarray}\label{equation:cor:rate:complexity:strong:convex2}
(\zeta^{-1})^T\le(\zeta^{-1})^{\log_{\frac{1}{\rho}}\left(\mathsf{A}\epsilon^{-1}\right)+1}
=\left(\frac{\rho}{\zeta}\right)^{\log_{\frac{1}{\rho}}\left(\mathsf{A}\epsilon^{-1}\right)+1}\cdot
\left(\frac{1}{\rho}\right)^{\log_{\frac{1}{\rho}}\left(\mathsf{A}\epsilon^{-1}\right)+1}
=\mathsf{A}\rho\epsilon^{-1}\left(\frac{\rho}{\zeta}\right)^{\log_{\frac{1}{\rho}}\left(\mathsf{A}\epsilon^{-1}\right)+1}.
\end{eqnarray}

From \eqref{equation:cor:rate:complexity:strong:convex1}-\eqref{equation:cor:rate:complexity:strong:convex2}, we prove the first claim of the corollary. If $\mu$, $\zeta$ and $\mathsf{a}\in(0,1]$ are chosen as stated in the corollary, we have $\rho=\zeta$ by definition of $\rho$ in Theorem  \ref{thm:rate:strong:convex}. From this, 
$
\zeta=1-\mathsf{a}\left(\mu\frac{c}{2L}+\phi\right)
$
and \eqref{equation:cor:rate:complexity:strong:convex1}-\eqref{equation:cor:rate:complexity:strong:convex2}, we prove the second statement of the corollary.\qed
\end{proof}

\begin{remark}[Constants for the strongly convex case]\label{remark:constants:strongly:convex}
As in Remark \ref{remark:constants:smooth:convex}, we compare the bounds of Theorem \ref{thm:rate:strong:convex} with previous bounds under \eqref{equation:uniform:variance:intro}. By the proof of Theorem \ref{thm:rate:strong:convex}, the constant $\mathsf{C}$ in Theorem \ref{thm:rate:strong:convex} is
\begin{eqnarray}
\mathsf{C}&:=&\left(1-\mu\frac{c}{L}\right)^{-1}\Vert x^1-x^*\Vert^2+\frac{4\mu}{(1-\mu)N_0Lc}\left\{\sigma_L^2\cdot\max_{\tau\in[t_0-1]}\esp\left[\Vert x^\tau-x^*\Vert^2\right]+2\sigma(x^*)^2\right\},\label{equation:c:strongly:convex}
\end{eqnarray} 
where $t_0$ is estimated by \eqref{equation:t0:strongly:convex}. As in the case of ill-conditioned smooth convex functions, we may have $t_0>1$. However, such dependence is milder since we do not have the factor $(t_0+1)^2$ and $t_0$ is \emph{logarithmic} with the endogenous and exogenous parameters. As a result, larger values of $t_0$ are acceptable. An interesting property in the case of strongly convex functions under Assumption \ref{assump:oracle:multiplicative:noise} is that the $\mathsf{L}^2$-boundedness and linear convergence of the generated sequence are obtained in the same proof. Next, we show that $t_0$ in \eqref{equation:t0:strongly:convex} is not too large in comparison to problems satisfying \eqref{equation:uniform:variance:intro}. We note that $t_0$ does not depend on any $x\in X$, but only on $\left(\frac{\sigma_L}{L}\right)^2$, $\phi:=\phi(\mu,\kappa)$ and the exogenous parameters $\mu$, $\zeta$ and $N_0$. Let us assume the standard Lipschitz continuity \eqref{equation:lipschitz:random:intro} of Lemma \ref{lemma:multiplicative:noise:from:random lipschitz} so that $\sigma_L=2L$ for
$
L:=\Lnorm{\mathsf{L}(\xi)},
$
and $\alpha\le\frac{\mu}{L}$. Without loss on generality, we may set $\phi:=1/(2\kappa$) in Theorem \ref{thm:rate:strong:convex} obtaining 
\begin{eqnarray}
t_0=\left\lceil\log_{\frac{1}{\zeta}}\left(\frac{8\mu^2\kappa}{(1-\mu)N_0}\right)\right\rceil\bigvee1,
\label{equation:t0:specific}
\end{eqnarray}
from \eqref{equation:t0:strongly:convex}. Thus $t_0=\mathcal{O}(1)\ln(\kappa/N_0)$ for a given set of exogenous parameters.

We thus conclude: the iteration $t_0$ is dictated solely by the multiplicative per unit distance variance $\sigma_L^2$ and the condition number $\kappa$, independently of the variances $\{\sigma(x)^2\}_{x\in X}$ at the points of the feasible set $X$. Moreover, assuming $L$ is known for the stepsize policy, there exists an upper bound on $t_0$ which is also independent of $\sigma_L^2$ and that only depends on $\ln(\kappa)$ and the exogenous parameters $\mu$, $N_0$, $\zeta$ defined in the stepsize and sampling rate policies. 

The proof of Theorem \ref{thm:rate:strong:convex} also says that for all $t\ge t_0$,
$
\esp\left[\Vert x^{t+1}-x^*\Vert^2\right]\le\mathsf{C}_0\rho^{t+1}
$
and for all $t\ge t_0+1$,
\begin{eqnarray}
\esp\left[g(x^{t+1})-g^*\right]&\le &\left[\frac{\left(L\mu^{-1}-c\right)}{2}+\frac{2\mu}{(1-\mu)N_0}\cdot\frac{\sigma_L^2}{L}\zeta^t\right]\mathsf{C}_0\rho^t+\frac{2\mu}{(1-\mu)N_0}\cdot\frac{\sigma(x^*)^2}{L}\zeta^t,\label{equation:linear:rate:optimality:from:t0+1}
\end{eqnarray}
where
\begin{eqnarray}
\mathsf{C}_0&:=&\left(1-\mu\frac{c}{L}+\phi\right)^{-t_0}\esp\left[\Vert x^{t_0}-x^*\Vert^2\right]+\frac{4\mu\sigma(x^*)^2}{(1-\mu)N_0Lc}.\label{equation:c0:strongly:convex}
\end{eqnarray}
If we compare the constant $\mathsf{C}$ in \eqref{equation:c:strongly:convex} with the constant $\mathsf{C}_0$ in \eqref{equation:c0:strongly:convex}, we note that if $t_0>1$ then $\mathsf{C}_0$ has a larger factor $\left(1-\mu\frac{c}{L}+\phi\right)^{-t_0}$ when compared to $\left(1-\mu\frac{c}{L}+\phi\right)^{-1}$ but $\mathsf{C}_0$ does not have the additional term $\frac{\sigma_L^2}{Lc}\cdot\max_{\tau\in[t_0-1]}\esp[\Vert x^\tau-x^*\Vert^2]$ found in $\mathsf{C}$. This last term is of the order of $\kappa\cdot\max_{\tau\in[t_0-1]}\esp[\Vert x^\tau-x^*\Vert^2]$ and may be larger than $\frac{\sigma(x^*)^2}{Lc}$.

Based on \eqref{equation:t0:specific}-\eqref{equation:c0:strongly:convex}, there are two different regimes depending when $\kappa$ is known or not. For instance, if an upper bound $\kappa_1$ of $\kappa$ is known, we get $t_0=1$ if we set, e.g.,
$
0<\mu\le\frac{-e+\sqrt{e^2+32e\kappa_1}}{16\kappa_1}=\mathcal{O}(1/\sqrt{\kappa_1}),
$
$N_0:=1$ and $\zeta:=e^{-1}$.
This policy puts more effort on the stepsize than on the sampling rate, with a convergence rate of order $\left[1-\frac{\mathcal{O}(1)}{\sqrt{\kappa_1}\kappa}\right]^t$. We can also get $t_0=1$ by setting, e.g., 
$
0<\mu\le\frac{-e+\sqrt{e^2+32e}}{16}=\mathcal{O}(1),
$
$N_0:=\kappa_1$ and $\zeta:=e^{-1}$. This policy puts more effort on the sampling rate; hence it has a factor of $\mathcal{O}(\kappa_1)$ on the oracle complexity but with a faster convergence rate of $\left[1-\frac{\mathcal{O}(1)}{\kappa}\right]^t$. In both cases, since $t_0=1$, the linear convergence stated by \eqref{equation:linear:rate:optimality:from:t0+1}-\eqref{equation:c0:strongly:convex} is an improvement when compared to a policy in which $t_0>1$ since for $t_0=1$ there is no dependence on $\kappa\cdot\max_{\tau\in[t_0-1]}\esp[\Vert x^\tau-x^*\Vert^2]$. When $t_0=1$ (given an upper bound of $\kappa$ and parameters $\mu$, $N_0$ and $\zeta$), the bound given by \eqref{equation:linear:rate:optimality:from:t0+1}-\eqref{equation:c0:strongly:convex} states that for all $t\ge2$,
$$
\esp\left[g(x^{t+1})-g^*\right]\lesssim\left[L\Vert x^1-x^*\Vert^2+\frac{\sigma(x^*)^2}{c}\right]\rho^{t-1},
$$
where we only considered the dominant terms (ignoring $-c$ and the decaying terms of $\mathcal{O}(\zeta^t)$). The above bound has the following property (\textsf{P}): it resembles the bound obtained for the strongly convex case if it was supposed that \emph{\eqref{equation:uniform:variance:intro} holds but replacing $\frac{\sigma^2}{c}$ with $\frac{\sigma(x^*)^2}{c}$} (see\footnote{The result in \cite{byrd:chiny:nocedal:wu2012}, Theorem 4.2 is for $\varphi\equiv0$ and $X=\re^d$ so the bound
$
f(x^1)-f^*\le\frac{L}{2}\Vert x^1-x^*\Vert^2
$
holds by Lemma \ref{lemma:inner:upper:approximation:inequality}.
For $\varphi\equiv0$ and $X=\re^d$, we could obtain
$
\esp\left[f(x^{t+1})-f^*\right]\lesssim\left[f(x^1)-f^*+\frac{\sigma(x^*)^2}{c}\right]\rho^{t-1}.
$} \cite{byrd:chiny:nocedal:wu2012}, Theorem 4.2). In this sense, we improve previous results by showing that under the more aggressive setting of Assumption \ref{assump:oracle:multiplicative:noise}, our bounds depends on the \emph{local variance $\sigma(x^*)^2$}. Moreover, in case \eqref{equation:uniform:variance:intro} indeed holds, our bounds are sharper than in \cite{byrd:chiny:nocedal:wu2012} since typically $\sigma(x^*)^2\ll\sigma^2$ for large $\sqrt{\diam(X)}$ (see Example \ref{example:1}). This may be seen as a \emph{localization property} of \textsf{Algorithm \ref{algorithm:smooth:strongly:convex}} in terms of the oracle's variance.

Suppose now that $\kappa$ cannot be estimated. We have the mild dependence $t_0\sim\mathcal{O}(\ln(\mu^2\kappa/(1-\mu)N_0)\vee1$ with $\kappa$. Hence, $t_0\rightarrow1$ as either $\mu$ decreases, $N_0$ increases or $\kappa$ decreases. We thus conclude that property (\textsf{P}) tends to be satisfied for better conditioned problems or for bigger initial batch sizes. Also, the bounds given by \eqref{equation:linear:rate:optimality:from:t=2} and \eqref{equation:c:strongly:convex} depend on the local variance estimation $\sigma(x^*)^2+\sigma_L^2\max_{\tau\in[t_0-1]}\esp[\Vert x^\tau-x^*\Vert^2]$. Moreover, if \eqref{equation:uniform:variance:intro} indeed holds, our rate statements given by \eqref{equation:linear:rate:optimality:from:t=2} and \eqref{equation:c:strongly:convex} are also sharper than in \cite{byrd:chiny:nocedal:wu2012} since typically $\sigma(x^*)^2+\sigma_L^2\max_{\tau\in[t_0-1]}\esp[\Vert x^\tau-x^*\Vert^2]\ll\sigma^2$ if $\max_{\tau\in[t_0-1]}\esp[\Vert x^\tau-x^*\Vert^2]\ll\diam(X)^2$. See Example \ref{example:1}. This is again a localization property of \textsf{Algorithm \ref{algorithm:smooth:strongly:convex}} in terms of the oracle's variance.
\end{remark}

\begin{remark}[Robust sampling]
From Theorems \ref{thm:rate:smooth:convex:multiplicative:noise}-\ref{thm:rate:strong:convex}, $\{\alpha_t\}$ and $\{N_t\}$ do not require knowledge of $\{\sigma^2(x)\}_{x\in X}$. Precisely, if $a$ and $N_0$ are not tuned to $\sigma(x^*)^2$ for $x^*\in\solset(f,\varphi)$, then the algorithm keeps running with proportional scaling in the convergence rate and oracle complexity. In this sense, the dynamic mini-batch scheme we propose is robust (see \cite{iusem:jofre:oliveira:thompson2017,nem:jud:lan:shapiro2009} for comments on robust methods). 
\end{remark}
	
\section*{Appendix: proofs of Lemmas \ref{lemma:multiplicative:noise:from:random lipschitz} and \ref{lemma:oracle:variance:decay}}
\begin{proof}[Proof of Lemma \ref{lemma:multiplicative:noise:from:random lipschitz}]
Let $x,y\in X$. Jensen's inequality and \eqref{equation:lipschitz:random:intro} imply
\begin{eqnarray*}
\Vert\nabla f(x)-\nabla f(y)\Vert\le\esp\left[\Vert \nabla F(x,\xi)-\nabla F(y,\xi)\Vert\right]
\le\esp[\mathsf{L}(\xi)]\Vert x-y\Vert.
\end{eqnarray*}
The first claim is proved using the above and $\esp[\mathsf{L}(\xi)]\le\Lnorm{\mathsf{L}(\xi)}$ by H\"older's inequality. Using this, we get
\begin{eqnarray*}
\sigma(x)&\le &\Lnorm{\Vert\nabla F(x,\xi)-\nabla F(y,\xi)\Vert}+\Lnorm{\Vert\nabla F(y,\xi)-\nabla f(y)\Vert}+\Lnorm{\Vert\nabla f(x)-\nabla f(y)\Vert}\\
&\le &\Lnorm{\mathsf{L}(\xi)\Vert x-y\Vert}+\sigma(y)+L\Vert x-y\Vert=\sigma(y)+2L\Vert x-y\Vert,\end{eqnarray*}
where we used the triangle inequality for $\Vert\cdot\Vert$ and Minkowski's inequality for $\Lnorm{\cdot}$.
\qed
\end{proof}

\begin{proof}[Proof of Lemma \ref{lemma:oracle:variance:decay}]
Let $x\in\re^d$. Since $\{\xi_j\}_{j=1}^N$ is i.i.d. and \eqref{equation:expected:valued:objective} holds, the sequence $\{\epsilon_j\}_{j=1}^N$ defined by
$
\epsilon_j:=\frac{\nabla F(x,\xi_j)-\nabla f(x)}{N}
$
is an i.i.d. sequence of random vectors with zero mean. As a consequence,\footnote{To show this, let
$\epsilon_i[\ell]$ denote the $\ell$-th coordinate of the vector $\epsilon_i$. From the Pythagorean identity and linearity of the expectation, we get 
$
\esp\left[\left\Vert\sum_{j=1}^N\epsilon_j\right\Vert^2\right]
=\sum_{j=1}^N\esp\left[\Vert\epsilon_j\Vert^2\right]+2\sum_{i<j}\esp\left[\langle\epsilon_i,\epsilon_j\rangle\right]
=\sum_{j=1}^N\esp\left[\Vert\epsilon_j\Vert^2\right]+2\sum_{i<j}\sum_{\ell=1}^d\esp\left\{\epsilon_i[\ell]\epsilon_j[\ell]\right\}.
$
The claim follows immediately from $\esp\left\{\epsilon_i[\ell]\epsilon_j[\ell]\right\}=\esp\left\{\epsilon_i[\ell]\right\}\esp\left\{\epsilon_j[\ell]\right\}=0$ for $i<j$, since for every $\ell\in[d]$, $\epsilon_i[\ell]$ and $\epsilon_j[\ell]$ are centered independent real-valued random variables.
}
we have 
$
\esp\left[\left\Vert\sum_{j=1}^N\epsilon_j(x)\right\Vert^2\right]=\sum_{j=1}^N\esp\left[\Vert\epsilon_j(x)\Vert^2\right].
$
Thus we get
\begin{eqnarray*}
\Lnorm{\Vert\epsilon(x)\Vert}=\sqrt{\sum_{j=1}^N\esp[\Vert\epsilon_j(x)\Vert^2]}
=\sqrt{\sum_{j=1}^N\frac{\sigma(x)^2}{N^2}}
&=&\frac{\sigma(x)}{\sqrt{N}},
\end{eqnarray*}
where in second equality we used that $\{\xi_j\}$ is drawn from $\probn$. The claim follows immediately from the above and Assumption \ref{assump:oracle:multiplicative:noise}.\qed
\end{proof}



\end{document}